\documentclass[12pt]{amsart}
\usepackage{mathrsfs}
\usepackage{amssymb}
\usepackage{verbatim}
\usepackage{hyperref}
\usepackage{color}
\usepackage{amsfonts}
\usepackage{mathrsfs}
\usepackage{amsmath}
\usepackage{amssymb}

\begin{document}
\newtheorem{theorem}{Theorem}
\newtheorem{proposition}[theorem]{Proposition}
\newtheorem{conjecture}[theorem]{Conjecture}
\def\theconjecture{\unskip}
\newtheorem{corollary}[theorem]{Corollary}
\newtheorem{lemma}[theorem]{Lemma}
\newtheorem{sublemma}[theorem]{Sublemma}
\newtheorem{observation}[theorem]{Observation}
\theoremstyle{definition}
\newtheorem{definition}{Definition}
\newtheorem{notation}[definition]{Notation}
\newtheorem{remark}[definition]{Remark}
\newtheorem{question}[definition]{Question}
\newtheorem{questions}[definition]{Questions}
\newtheorem{example}[definition]{Example}
\newtheorem{problem}[definition]{Problem}
\newtheorem{exercise}[definition]{Exercise}

\numberwithin{theorem}{section}
\numberwithin{definition}{section}
\numberwithin{equation}{section}

\def\earrow{{\mathbf e}}
\def\rarrow{{\mathbf r}}
\def\uarrow{{\mathbf u}}
\def\varrow{{\mathbf V}}
\def\tpar{T_{\rm par}}
\def\apar{A_{\rm par}}

\def\reals{{\mathbb R}}
\def\torus{{\mathbb T}}
\def\heis{{\mathbb H}}
\def\integers{{\mathbb Z}}
\def\naturals{{\mathbb N}}
\def\complex{{\mathbb C}\/}
\def\distance{\operatorname{distance}\,}
\def\support{\operatorname{support}\,}
\def\dist{\operatorname{dist}\,}
\def\Span{\operatorname{span}\,}
\def\degree{\operatorname{degree}\,}
\def\kernel{\operatorname{kernel}\,}
\def\dim{\operatorname{dim}\,}
\def\codim{\operatorname{codim}}
\def\trace{\operatorname{trace\,}}
\def\Span{\operatorname{span}\,}
\def\dimension{\operatorname{dimension}\,}
\def\codimension{\operatorname{codimension}\,}
\def\nullspace{\scriptk}
\def\kernel{\operatorname{Ker}}
\def\ZZ{ {\mathbb Z} }
\def\p{\partial}
\def\rp{{ ^{-1} }}
\def\Re{\operatorname{Re\,} }
\def\Im{\operatorname{Im\,} }
\def\ov{\overline}
\def\eps{\varepsilon}
\def\lt{L^2}
\def\diver{\operatorname{div}}
\def\curl{\operatorname{curl}}
\def\etta{\eta}
\newcommand{\norm}[1]{ \|  #1 \|}
\def\expect{\mathbb E}
\def\bull{$\bullet$\ }

\def\xone{x_1}
\def\xtwo{x_2}
\def\xq{x_2+x_1^2}
\newcommand{\abr}[1]{ \langle  #1 \rangle}

\newcommand{\Norm}[1]{ \left\|  #1 \right\| }
\newcommand{\set}[1]{ \left\{ #1 \right\} }
\def\one{\mathbf 1}
\def\whole{\mathbf V}
\newcommand{\modulo}[2]{[#1]_{#2}}
\def \essinf{\mathop{\rm essinf}}
\def\scriptf{{\mathcal F}}
\def\scriptg{{\mathcal G}}
\def\scriptm{{\mathcal M}}
\def\scriptb{{\mathcal B}}
\def\scriptc{{\mathcal C}}
\def\scriptt{{\mathcal T}}
\def\scripti{{\mathcal I}}
\def\scripte{{\mathcal E}}
\def\scriptv{{\mathcal V}}
\def\scriptw{{\mathcal W}}
\def\scriptu{{\mathcal U}}
\def\scriptS{{\mathcal S}}
\def\scripta{{\mathcal A}}
\def\scriptr{{\mathcal R}}
\def\scripto{{\mathcal O}}
\def\scripth{{\mathcal H}}
\def\scriptd{{\mathcal D}}
\def\scriptl{{\mathcal L}}
\def\scriptn{{\mathcal N}}
\def\scriptp{{\mathcal P}}
\def\scriptk{{\mathcal K}}
\def\frakv{{\mathfrak V}}
\def\C{\mathbb{C}}
\def\R{\mathbb{R}}
\def\Rn{{\mathbb{R}^n}}
\def\Sn{{{S}^{n-1}}}
\def\M{\mathbb{M}}
\def\N{\mathbb{N}}
\def\Q{{\mathbb{Q}}}
\def\Z{\mathbb{Z}}
\def\F{\mathcal{F}}
\def\L{\mathcal{L}}
\def\S{\mathcal{S}}
\def\supp{\operatorname{supp}}
\def\dist{\operatorname{dist}}
\def\essi{\operatornamewithlimits{ess\,inf}}
\def\esss{\operatornamewithlimits{ess\,sup}}
\author{Mingming Cao}
\address{Mingming Cao \\
         School of Mathematical Sciences \\
         Beijing Normal University \\
         Laboratory of Mathematics and Complex Systems \\
         Ministry of Education \\
         Beijing 100875 \\
         People's Republic of China}
\email{m.cao@mail.bnu.edu.cn}

\author{Qingying Xue}
\address{Qingying Xue\\
        School of Mathematical Sciences\\
        Beijing Normal University \\
        Laboratory of Mathematics and Complex Systems\\
        Ministry of Education\\
        Beijing 100875\\
        People's Republic of China}
\email{qyxue@bnu.edu.cn}

\thanks{The authors were supported partly by NSFC
(No. 11471041), the Fundamental Research Funds for the Central Universities (No. 2014KJJCA10) and NCET-13-0065. \\ \indent Corresponding
author: Qingying Xue\indent Email: qyxue@bnu.edu.cn}


\date{June 15, 2015.}
\title[Littlewood-Paley $g_{\lambda}^{*}$-function]{A non-homogeneous local $Tb$ theorem for Littlewood-Paley $g_{\lambda}^{*}$-function with $L^p$-testing condition}
\maketitle

\begin{abstract}
In this paper, we present a local $Tb$ theorem for the non-homogeneous Littlewood-Paley $g_{\lambda}^{*}$-function with non-convolution type kernels and upper power bound measure $\mu$. We show that, under the assumptions
$\supp b_Q \subset Q$, $|\int_Q b_Q d\mu| \gtrsim \mu(Q)$ and $||b_Q||^p_{L^p(\mu)} \lesssim \mu(Q)$, the norm inequality
$\big\| g_{\lambda}^{*}(f) \big\|_{L^p(\mu)} \lesssim \big\| f \big\|_{L^p(\mu)}$
holds if and only if the following testing condition holds :
$$
\sup_{Q : cubes \ in \ \Rn} \frac{1}{\mu(Q)}\int_Q \bigg(\int_{0}^{\ell(Q)} \int_{\Rn} \Big(\frac{t}{t+|x-y|}\Big)^{m\lambda}|\theta_t(b_Q)(y,t)|^2
\frac{d\mu(y) dt}{t^{m+1}}\bigg)^{p/2} d\mu(x) < \infty.
$$
This is the first time to investigate $g_\lambda^*$-function in the simultaneous presence of three attributes : local, non-homogeneous and $L^p$-testing condition. It is important to note that the testing condition here is $L^p$ type with
$p \in (1,2]$.
\end{abstract}

\section{Introduction}
The first local $Tb$ theorem was due to Christ \cite{C}. It was shown that if there exists functions $b_Q^1$ and $b_Q^2$ satisfying $||b_Q^1||_{L^\infty} \leq C$, $||b_Q^2||_{L^\infty} \leq C$, $||T b_Q^1||_{L^\infty} \leq C$ and $||T^* b_Q^2||_{L^\infty} \leq C$, then Calder\'{o}n-Zygmund operator $T$ is bounded on $L^2(X,\mu)$, where $X$ is a homogeneous space and $\mu$ is a doubling measure.
Later, Nazarov, Treil and Volberg \cite{NTV2002} proved that such accretive system does exist. It is not only in the nonhomogeneous situation, but also allows the operator to map the functions of the accretive system into $BMO$ space instead of $L^\infty$. It was the first time to discuss the non-homogeneous analysis in local situation. After that, the $L^2$ type testing conditions were introduced by Auscher, Hofmann, Muscalu, Tao and Thiele \cite{AHMTT}. The assumptions are of the form $\int_Q |b_Q^1|^2 \leq |Q|$, $\int_Q |b_Q|^2 \leq |Q|$, $\int_Q |T b_Q^1|^2 \leq |Q|$ and $\int_Q |T^* b_Q|^2 \leq |Q|$. But it was only proved for Lebesgue measure and perfect dyadic Calder\'{o}n-Zygmund operator. Hofmann \cite{Hofmann} extended it to general Calder\'{o}n-Zygmund operators. However, he needed a stronger conditions $\int_Q |b_Q^1|^s \leq |Q|$, $\int_Q |b_Q|^s \leq |Q|$, $\int_Q |T b_Q^1|^2 \leq |Q|$ and $\int_Q |T^* b_Q|^2 \leq |Q|$ for some $s>2$, which leaves much to be desired.

Recently, a local $Tb$ theorem for the general Calder\'{o}n-Zygmund operator in non-homogeneous space was presented by Hyt\"{o}nen and Martikainen \cite{HM}. The upper doubling measure and accretive $L^\infty$ systems were assumed. The authors extended and modified the general non-homogeneous technique of Nazarov, Treil and Volberg \cite{NTV2002} to their case. After that, a local $Tb$ theorem with non-scale-invariant $L^2$-testing conditions was proved by Lacey and Martikainen \cite{LM-CZ}. The combination of non-scale-invariance and general measures is a subtle issue. To handle these obstacles, they made full use of the techniques of non-homogeneous and two-weight dyadic analysis. In the proof, some complicated paraproducts were generated. This problem leaded to obtain critical twisted martingale difference inequalities.

In the square function setting, Martikainen, Mourgoglou and Orponen \cite{MMO} first obtained a non-homogeneous local $Tb$ theorem only with scale invariant $L^\infty$ testing conditions, which corresponds with the most general assumptions in \cite{NTV2002}. Moreover, the authors used the averaging identity over good Whitney regions, which makes the proof quite brief. Later, in the non-homogeneous world local $Tb$ theorems with $L^p$ type testing conditions , which are the non-scale-invariant case, were also established. The
$L^2$-testing conditions for the Littlewood-Paley $g$-function were given by Lacey and Martikainen \cite{LM-S}. The $L^p(1<p<2)$ type testing conditions for $g$-function were provided in \cite{MM} by Martikainen and Mourgoglou. The combination of the general testing functions with upper doubling measures leads some problems to be arisen. To overcome them, the twisted martingale difference operators associated with stopping cubes were introduced. In addition, as for the difficult part $p \in (1,2)$, there is more work to do. It includes an essential $p$-Carleson estimates, $T1$ theorem, the new construction of stopping cubes and square function estimates of martingale difference. Some techniques were learnt from \cite{H2014} and \cite{LM-CZ}.

Still more recently, Cao, Li and Xue \cite{CLX} gave a characterization of two weight norm inequalities for the classical Littlewood-Paley $g_\lambda^*$-function of higher dimension, which was first introduced and studied by Stein \cite{Stein1961} in 1961. The key to the proof was based on a new averaging identity over good Whitney regions. The identity is a further development of Hyt\"{o}nen's improvement \cite{H2012} of the Nazarov-Treil-Volberg method of random dyadic systems \cite{NTV2003}. In addition, the martingale difference decomposition was used. The core of the proof is the construction of stopping cubes with respect to a fractional version of Pivotal condition, which is a modern and effective technique to deal with two-weighted problems.

In this article, we continue the research of Littlewood-Paley $g_\lambda^*$-function but in the non-homogeneous situation. We are concerned with the local $Tb$ theorem with non-scale-invariant $L^p$-testing conditions for $1 < p \le 2$. Before we state the main theorem, it is necessary to give some definitions and notations. The generalized Littlewood-Paley $g_\lambda^*$-function is defined by
\begin{align*}
g_{\lambda}^*(f)(x)
= \bigg(\iint_{\R^{n+1}_{+}} \Big(\frac{t}{t + |x - y|}\Big)^{m \lambda} |\theta_t f(y)|^2
\frac{d\mu(y) dt}{t^{m+1}}\bigg)^{1/2},\ \lambda > 1,
\end{align*}
where $\theta_t f(y) = \int_{\Rn} s_t(y,z) f(z) d\mu(z)$, and the non-convolution type kernel $s_t$ satisfies the following Standard conditions:\\
\noindent\textbf{Standard conditions.} The kernel $s_t : \R^{n} \times \R^{n} \rightarrow \C$ is assumed to satisfy the following estimates: for some
$\alpha>0$
\begin{enumerate}
\item [(1)] Size condition :
$$ |s_t(x,y)| \lesssim \frac{t^{\alpha}}{(t + |x - y|)^{m+\alpha}}.$$
\item [(2)] H\"{o}lder condition :
$$ \quad\quad  \quad  \quad \ |s_t(x,y) - s_t(x,y') \lesssim \frac{|y - y'|^{\alpha}}{(t + |x - y|)^{m+\alpha}} ,$$
whenever $|y -y'| < t/2$ .
\end{enumerate}
Moreover, the non-homogeneous measure $\mu$ is a Borel measure on $\Rn$ satisfying the $upper \ power \ bound$: for some $m>0$,
$$ \mu(B(x,r)) \lesssim r^m, \ \ x \in \Rn, \ r>0.$$
In this paper, we will use $\ell^\infty$ metrics on $\Rn$. (Actually, our result is still true for the usual Euclidian distance metrics on $\Rn$ if we slightly modify the Standard conditions.) Here, we formulate the main result of this paper as follows:
\begin{theorem}\label{Theorem}
Let $1 < p \leq 2$, $\lambda > 2$, $ 0 < \alpha \leq m(\lambda - 2)/2 $. Assume that the kernel $s_t$ satisfies the Standard conditions, and for any cube
$Q \subset \Rn$ there exists a function $b_Q$ satisfying that
\begin{enumerate}
\item [(1)] $supp \ b_Q \ \subset Q$;
\item [(2)] $|\int_{Q} b_Q d\mu| \gtrsim \mu(Q)$;
\item [(3)] $||b_Q||_{L^p(\mu)}^p \lesssim \mu(Q)$;
\item [(4)] $$\mathcal{G}:=\sup_{\substack{Q \subset \Rn \\ Q: cube}}\bigg[\frac{1}{\mu(Q)}\int_Q \bigg(\int_{0}^{\ell(Q)} \int_\Rn
    \Big(\frac{t}{t+|x-y|}\Big)^{m\lambda} |\theta_t b_Q(y)|^2 \frac{d\mu(y) dt}{t^{m+1}} \bigg)^{\frac p2} d\mu(x)\bigg]^{\frac1p} < \infty.$$
\end{enumerate}
Then we have
$$||g_{\lambda}^*(f)||_{L^p(\mu)} \lesssim (1+\mathcal{G})||f||_{L^p(\mu)} .$$
\end{theorem}
\begin{remark}If we take simply $d\mu(x)=dx$ and $s_t(y,z)=p_t(y-z)$, where $p$ is the classical Poisson kernel, then the above operator coincides with the $g_{\lambda}^*$-function defined and studied by Stein \cite{Stein1961} in 1961 and later studied by Fefferman \cite{Fe} in 1970, Muckenhoupt and Wheeden \cite{MR} in 1974, etc. It should be noted that even for this classical case, the results in this paper are also completely new.
\end{remark}
To better understand a pseudo-accretive system satisfying the above conditions $(1)$, $(2)$ and $(3)$, we give some examples.

\vspace{0.3cm}
{\textbf{Examples.}}
Let $Q \subset \Rn$ be a cube, and $c_Q$ be the center of $Q$. When $\mu$ is the Lebesgue measure on $\Rn$, the following functions $b_Q$ satisfy the
conditions $(1)$, $(2)$ and $(3)$ in Theorem $\ref{Theorem}$.
\begin{enumerate}
\item [$\bullet$]{Characteristic  function}: $b_Q(x) =\mathbf{1}_{Q}(x)$. \\
\item [$\bullet$]\text{Gaussian  function}:  $b_Q(x) = \pi^{-\frac{n}{2}} \exp \big( -\frac{|x-c_Q|^2}{\ell(Q)^2}\big)\mathbf{1}_{Q}(x)$. \\
\item [$\bullet$]\text{Poisson kernel}:    $ b_Q(x) = \frac{\ell(Q)^{n+1}}{(\ell(Q)^2 + |x-c_Q|^2)^{\frac{n+1}{2}}} \mathbf{1}_{Q}(x)$. \\
\item [$\bullet$]\text{Polish function}:    $b_Q(x) =\exp \big(-\frac{\ell(Q)^2}{\ell(Q)^2-4|x-c_Q|^2}\big)\mathbf{1}_{Q}(x)$. \\
\item [$\bullet$]\text{$C_c^{\infty}$ function}:   $ b_Q(x) =\Psi_Q(|x-c_Q|^2)$,
where $\Psi_Q(t) = \int_{t}^{+\infty} \psi_Q(s) ds \Big/ \int_{-\infty}^{+\infty} \psi_Q(s) ds,$
and
$\psi_Q(t) = \exp \big( - \frac{1}{(t-a_Q^2)(d_Q^2-t)}\big) \mathbf{1}_{(a_Q,d_Q)}(t), \ \ a_Q=\ell(Q)/4, \ d_Q=\ell(Q)/2.$
It can be checked that $b_Q \in C_c^{\infty}(\Rn)$, $0 \leq b_Q \leq 1$, $\supp b_Q \subset Q$ and
\begin{equation*}
b_Q(x)=
\begin{cases}
 1 \ \ & x \in \frac{1}{2}Q . \\
 0 \ \ & x \notin Q.
\end{cases}
\end{equation*}
\end{enumerate}

Next let us discuss the general strategy and some new aspects of the proof of the main Theorem. It should be noted that some basic ideas are taken from \cite{CLX} and \cite{MM}. One of the most important steps is the reductions of our main theorem. It is first reduced to a priori assumption, which is always convenient to our analysis. Then we are reduced to proving $T1$ theorem $(\ref{key-1})$ and Carleson estimate $(\ref{key-2})$. Carleson estimate $(\ref{key-2})$ is the core of this paper, since the proof of $T1$ theorem $(\ref{key-1})$ except for a paraproduct estimate is largely contained in Carleson estimate $(\ref{key-2})$. Furthermore, the significant reason why we need Carleson estimate is that the testing condition here is $L^p$ type with $p \in (1,2)$, which means that the averaging identity over good Whitney regions we used in the previous articles are not suitable for our setting now. Beyond that, stopping cubes are constructed by means of testing function $b_Q$. Finally, the main technique of our proof is the using of twisted (adapted) martingale differences operators. Though they are more complicated than the classical case, they still have many good properties, for example, they have orthogonality and they are not only supported on dyadic cubes but also are constant on each of their children. The twisted martingale difference operators have been used and discussed by many authors in their proofs of different types of $Tb$ theorems, such as Nazarov, Treil and Volberg \cite{NTV2003}, Hyt\"{o}nen, Martikainen \cite{HM-1}, Lacey, Martikainen \cite{LM-CZ} and \cite{LM-S}.

\section{Some  Preliminaries}\label{Sec-Reduction}
In this section, our goal is to introduce some fundamental tools including the random dyadic grids and the twisted martingale difference operators. Additionally, several key lemmas are needed below.
\subsection{Random dyadic grids}
Here the random dyadic grids, which essentially go back to \cite{NTV2003}, are improved by Hyt\"{o}nen \cite{H2012,
 H2014}.
Let $\mathcal{D}^0$ denote the standard dyadic grid. That is,
$$ \mathcal{D}^0 := \bigcup_{k \in \Z}\mathcal{D}_k^0, \ \ \mathcal{D}_k^0 := \big\{2^{k}([0, 1)^N + m); k \in \Z, \ m \in \Z^n \big\}.$$
For a binary sequence $\beta = \{ \beta_j\}_{j \in \Z}$, where $\beta_j \in \{0,1\}^n$, we define
$$ \mathcal{D}_k^\beta := \Big\{I + \sum_{j:j<k} 2^{j} \beta_j; I \in \mathcal{D}_k^0 \Big\}.$$
Then we will get the general dyadic systems of the form
$$
\mathcal{D}^\beta := \bigcup_{k \in \Z}\mathcal{D}_k^\beta.
$$
To facilitate, we write $$ I + \beta := I + \sum_{j:j<k} 2^{j} \beta_j.$$
According to the canonical product probability measure $\mathbb{P}_\beta$ on $(\{0, 1\}^n)^\Z$ which makes the coordinates $\beta_j$ independent and identically distributed with $\mathbb{P}_\beta(\beta_j = \eta) = 2^{-n}$ for all $\eta \in \{0,1\}^n$.

A cube $I \in \mathcal{D}$ is said to be $bad$ if there exists a $J \in \mathcal{D}$ with $\ell(J) \geq 2^r \ell(I)$ such that
$\dist(I,\partial J) \leq \ell(I)^{\gamma} \ell(J)^{1-\gamma}$. Otherwise, $I$ is called $good$.
Here $r \in \Z_+$ and $\gamma \in (0,\frac12)$ are given parameters. Denote
$\pi_{good} = \mathbb{P}_{\beta}(I + \beta \ \text{is \ good}) = \mathbb{E}_{\beta}(\mathbf{1}_{good}(I+\beta))$.
Then $\pi_{good}$ is independent of $I \in \mathcal{D}_0$. And the parameter $r$ is a fixed constant so that $\pi_{good} > 0$.
Moreover, in this article, we fix the constant $\gamma$ to be so small that
$$\gamma \leq \frac{\alpha}{2(m+\alpha)} \ \ \text{and} \ \ \frac{m \gamma}{1-\gamma}\leq \frac{\alpha}{4},$$
where $\alpha > 0$ appears in the kernel estimates. And the parameter $r$ will be demanded to be sufficiently large as below.
It is important to observe that the position and goodness of a cube $I \in \mathcal{D}^0$ are independent (see \cite{H2012}).
\subsection{Stopping cubes.}
We shall present the construction of stopping cubes. For a small convenience we may assume the
normalisation $\langle b_Q \rangle_Q = 1$. Let $\mathcal{D}$ be a dyadic grid in $\Rn$ and let $Q^* \in \mathcal{D}$
be a fixed dyadic cube with $\ell(Q^*) = 2^s$. Set $\mathcal{F}^0_{Q^*} = \{ Q^* \}$ and let $\mathcal{F}^1_{Q^*}$ consist of
the maximal cubes $Q \in \mathcal{D}$, $Q \subset Q^*$, for which at least one of the following two conditions holds:
\begin{enumerate}
\item [(1)] $|\langle  b_{Q^*} \rangle_Q| < 1/2$;
\item [(2)] $\langle  |b_{Q^*}|^p \rangle_Q > 2^{p'+1} A^{p'}$.
\end{enumerate}
Here $A$ is a constant such that $||b_R||_{L^p(\mu)}^p \leq A \mu(R)$ for every cube $R \subset \Rn$. Next, we repeat the previous procedure by replacing $Q^*$ with a fixed $Q \in \mathcal{F}^1_{Q^*}$. The combined collection of stopping cubes resulting from this is called $\mathcal{F}^2_{Q^*}$. This is continued and
we set $\mathcal{F}_{Q^*} =\bigcup_{j=0}^\infty \mathcal{F}^j_{Q^*}$. Finally, for every $Q \in \mathcal{D}$, $Q \subset Q^*$, we let $Q^a \in \mathcal{F}_{Q^*}$ be the minimal cube $R \in \mathcal{F}_{Q^*}$ for which $Q \subset R$.

The notion of stopping cubes and the following basic facts can be found in \cite{MM}.
\begin{lemma}\label{F-Fj-Q*}
If $F \in \mathcal{F}^j_{Q^*}$ for some $j \geq 0$, then there holds that
$$\sum_{\substack{S \in \mathcal{F}^{j+1}_{Q^*} \\ S \subset F}} \mu(S) \lesssim \mu(F).$$
\end{lemma}
\begin{lemma}\label{Carleson-sequence}
Write
\begin{equation*}
a_Q =
\begin{cases}
\mu(Q), \ \ & Q \in \bigcup_j \mathcal{F}^j_{Q^*}; \\
0, \ \ & \text{otherwise}.
\end{cases}
\end{equation*}
Then $\{a_Q\}$ is a Carleson sequence. That is, for each dyadic cube $R$, it holds that
$$ \sum_{Q \subset R} a_Q \lesssim \mu(R).$$
\end{lemma}
\subsection{Twisted martingale difference operators.}
If $Q \in \mathcal{D}$, $Q \subset Q^*$, and $f \in L^1_{loc}(\mu)$, we define the twisted martingale difference operators as follows
$$ \Delta_Q f = \sum_{Q' \in ch(Q)} \Big[ \frac{\langle f \rangle_{Q'}}{\langle b_{(Q')^a} \rangle_{Q'}}b_{(Q')^a} -
\frac{\langle f \rangle_{Q}}{\langle b_{Q^a} \rangle_{Q}}b_{Q^a} \Big] \mathbf{1}_{Q'}.$$
Moreover, on the largest $Q^*$ level we consider $\Delta_{Q^*}$ as $\Delta_{Q^*} + \langle f \rangle_{Q^*} b_{Q^*}$. Thus, it holds that $\int \Delta_Q f d\mu
=0$ if $Q \subsetneq Q^*$. We also define
$$ \Delta_k f = \sum_{Q \in \mathcal{D}_k:Q \subset Q^*}\Delta_Q f.$$
The next vector-valued inequality was proved by Stein \cite[p.~103]{Stein1970}.
\begin{lemma}\label{Stein}
Let $(\mathcal{M}, \nu)$ be a $\sigma$-finite measure space and let $\mathfrak{M}$ denote the family of measurable subsets of $\mathcal{M}$. Suppose that $\{\mathcal{F}_k\}_k$is an infinite increasing sequence of $(\sigma-finite)$ $\sigma$-subalgebras of  $\mathfrak{M}$. Let $E_k = E(\cdot|\mathcal{F}_k)$ denote the conditional expectation operator with respect to $\mathcal{F}_k$. Assume that $\{f_k\}_k$ is any sequence of functions on $(\mathcal{M}, \nu)$, where $f_k$ is not assumed to be $\mathcal{F}_k$-measurable, and let $\{n_k\}_k$ be any sequence of positive integers. Then there holds that
$$
\Big\| \Big(\sum_{k \geq 1}|E_{n_k}f_k|^2\Big)^{1/2} \Big\|_{L^p(\nu)} \lesssim \Big\| \Big(\sum_{k \geq 1}|f_k|^2\Big)^{1/2} \Big\|_{L^p(\nu)}, \ 1 < p < \infty,
$$
where the implied constant defends only on $p$.
\end{lemma}
We also need the following martingale estimate, which has been proved in \cite{MM}.

\begin{lemma}\label{Delta-k}
Let $|f| \leq 1$. Then there holds that
$$ \Big\|\Big(\sum_{k}|\Delta_k f|^2\Big)^{1/2}\Big\|_{L^p(\mu)}^p \lesssim \mu(Q^*).$$
\end{lemma}
\section{Reductions of The Proof}
\subsection{Reduction to a priori assumption.}
\begin{proposition}
If under the additional priori assumption $||g_{\lambda}^*||_{L^p(\mu) \rightarrow L^p(\mu)} < \infty$ we obtain
$||g_{\lambda}^*||_{L^p(\mu) \rightarrow L^p(\mu)} \lesssim 1+\mathcal{G}$, then Theorem $\ref{Theorem}$ immediately
follows without a priori assumption.
\end{proposition}
\begin{proof}
We write the truncation of the kernel $s_t(x,y)$ as $s_t^i(x,y)=s_t(x,y)\mathbf{1}_{[1/i,i]}(t)$, for any $i > 0$. It is easy to check that $\{s_t^i\}_{i>0}$
satisfy the Standard condition. We write the truncation of Littlewood-Paley $g_\lambda^*$-function as
\begin{align*}
g_{\lambda,i}^*(f)(x)
= \bigg(\iint_{\R^{n+1}_{+}} \Big(\frac{t}{t + |x - y|}\Big)^{m \lambda} |\theta^i_t f(y)|^2
\frac{d\mu(y) dt}{t^{m+1}}\bigg)^{1/2},\ \lambda > 1,
\end{align*}
where the linear form
$$\theta_t^i f(x) = \int_{\Rn} s_t^i(x,y) f(y) d\mu(y).$$
We claim that for any $x \in \Rn$ and $t>0$ we have the pointwise control
\begin{equation}\label{pointwise}
H(x):= \bigg(\int_{\Rn} \Big(\frac{t}{t + |x - y|}\Big)^{m \lambda} |\theta_t f(y)|^2
\frac{d\mu(y)}{t^m}\bigg)^{1/2}
\lesssim M_{\mu} f(x),
\end{equation}
where $M_{\mu}$ is the centered maximal function
$$ M_{\mu}f(x) = \sup_{r>0} \frac{1}{\mu(B(x,r))} \int_{B(x,r)} |f(y)| d\mu(y).$$
This implies that
$$ ||g_{\lambda,i}^*(f)||_{L^p(\mu)} \lesssim (2 \log i)^{1/2} ||M_{\mu}f||_{L^p(\mu)} \lesssim (2 \log i)^{1/2} ||f||_{L^p(\mu)} < \infty.$$
Accordingly, it yields that
\begin{align*}
||g_{\lambda}^*(f)||_{L^p(\mu)}
&=\lim_{i \rightarrow \infty} ||g_{\lambda,i}^*(f)||_{L^p(\mu)}
\leq \lim \sup_{i \rightarrow \infty} ||g_{\lambda,i}^*||_{L^p(\mu) \rightarrow L^p(\mu)} ||f||_{L^p(\mu)}\\
&\lesssim \lim \sup_{i \rightarrow \infty} (1+\mathcal{G}_i) ||f||_{L^p(\mu)}
\leq (1+\mathcal{G}) ||f||_{L^p(\mu)}.
\end{align*}

It only remains to prove the inequality $(\ref{pointwise})$. Indeed, using the size condition of the kernel, one can get
$$ H(x) \lesssim \bigg[\int_{\Rn}\bigg(\int_{\Rn}\frac{t^\alpha}{(t+|y-z|)^{m+\alpha}}|f(z)|d\mu(z)\bigg)^2 \Big(\frac{t}{t + |x - y|}\Big)^{m
\lambda}\frac{d\mu(y)}{t^m}\bigg]^{1/2}.$$
Thus, it suffices to bound the following two parts:
$$ H_1(x) := \bigg[\int_{\Rn}\bigg(\int_{|z-x| > 2|y-x|}\frac{t^\alpha}{(t+|y-z|)^{m+\alpha}}|f(z)|d\mu(z)\bigg)^2 \Big(\frac{t}{t + |x - y|}\Big)^{m
\lambda}\frac{d\mu(y)}{t^m}\bigg]^{1/2},$$
and
$$ H_2(x) := \bigg[\int_{\Rn}\bigg(\int_{|z-x|\leq 2|y-x|}\frac{t^\alpha}{(t+|y-z|)^{m+\alpha}}|f(z)|d\mu(z)\bigg)^2 \Big(\frac{t}{t + |x - y|}\Big)^{m
\lambda}\frac{d\mu(y)}{t^m}\bigg]^{1/2}.$$
For $H_1$, it holds that $|y-z| \geq |x-z| - |x-y| \gtrsim |x-z|$. Hence, it yields
$$ H_1(x) \lesssim \int_{\Rn}\frac{t^\alpha}{(t+|x-z|)^{m+\alpha}}|f(z)|d\mu(z) \lesssim M_{\mu}f(x).$$
For $H_2$, we have for $0 < \alpha \leq m(\lambda-2)/2$ that
$$ \Big(\frac{t}{t + |x - y|}\Big)^{m \lambda/2} \leq \Big(\frac{t}{t + |x - y|}\Big)^{m +\alpha} \lesssim \frac{t^{m+\alpha}}{(t + |x - z|)^{m+\alpha}}. $$
Combining with Young's inequality, this gives us that
\begin{align*}
H_2(x)
&\lesssim \bigg[\int_{\Rn}\bigg(\int_{|z-x|\leq 2|y-x|}\frac{t^\alpha}{(t+|y-z|)^{m+\alpha}}\frac{t^{m+\alpha}}{(t + |x - z|)^{m+\alpha}}|f(z)|d\mu(z)\bigg)^2
\frac{d\mu(y)}{t^m}\bigg]^{1/2} \\
&\lesssim t^{m/2} || \phi*\psi_x ||_{L^2(\mu)}
\leq t^{m/2} || \phi ||_{L^2(\mu)} || \psi_x ||_{L^1(\mu)}
\lesssim || \psi_x ||_{L^1(\mu)} \lesssim M_{\mu}f(x),
\end{align*}
where
$$ \phi(z)= \frac{t^\alpha}{(t+|z|)^{m+\alpha}}, \ \ \psi_x(z)= \frac{t^\alpha}{(t+|x-z|)^{m+\alpha}}|f(z)|.$$
\end{proof}
\subsection{Reduction to a Carleson estimate.} We will further make reductions under the priori assumption $||g_{\lambda}^*||_{L^p(\mu) \rightarrow L^p(\mu)} < \infty$,  For any
$\kappa > 0$, we define
$$\mathcal{G}_{\text{loc}}(\kappa):=\sup_{\substack{Q \subset \Rn \\ Q: cube}}\bigg[\frac{1}{\mu(\kappa Q)}\int_Q \bigg(\int_{0}^{\ell(Q)} \int_\Rn
\Big(\frac{t}{t+|x-y|}\Big)^{m\lambda} |\theta_t \mathbf{1}_{Q}(y)|^2 \frac{d\mu(y) dt}{t^{m+1}} \bigg)^{p/2} d\mu(x)\bigg]^{1/p},$$

$$\mathcal{G}_{\text{glo}}(\kappa):=\sup_{\substack{Q \subset \Rn \\ Q: cube}}\bigg[\frac{1}{\mu(\kappa Q)}\int_Q \bigg(\int_{0}^{\ell(Q)} \int_\Rn
\Big(\frac{t}{t+|x-y|}\Big)^{m\lambda} |\theta_t \mathbf{1}(y)|^2 \frac{d\mu(y) dt}{t^{m+1}} \bigg)^{p/2} d\mu(x)\bigg]^{1/p}.$$
Once it has been proved
\begin{equation}\label{key-1}
||g_{\lambda}^*||_{L^p(\mu) \rightarrow L^p(\mu)} \leq C_1 (1+\mathcal{G}_{\text{glo}}(9)) \leq C_2 (1+\mathcal{G}_{\text{loc}}(3)),
\end{equation}
and
\begin{equation}\label{key-2}
\mathcal{G}_{\text{loc}}(3) \leq C_3 (1+\mathcal{G}) + C_2^{-1} ||g_{\lambda}^*||_{L^p(\mu) \rightarrow L^p(\mu)}/2,
\end{equation}
we shall get
$$
||g_{\lambda}^*||_{L^p(\mu) \rightarrow L^p(\mu)} \leq C (1+\mathcal{G}) + ||g_{\lambda}^*||_{L^p(\mu)\rightarrow L^p(\mu)}/2.
$$
Thus, we are reduced to proving $(\ref{key-1})$ and $(\ref{key-2})$. In section $\ref{sec-T1}$, we will see that the proof of $(\ref{key-1})$ is largely
similar to that of $(\ref{key-2})$. Hence, we focus on showing $(\ref{key-2})$.
\subsection{Reduction to good cubes.}\label{reduction-good}
In order to obtain $(\ref{key-2})$, it suffices to bound
$$ \bigg[\int_{Q_0} \bigg(\int_{0}^{\ell(Q_0)} \int_\Rn \Big(\frac{t}{t+|x-y|}\Big)^{m\lambda} |\theta_t f(y)|^2 \frac{d\mu(y) dt}{t^{m+1}} \bigg)^{p/2}
d\mu(x)\bigg]^{1/p} $$
for every given cube $Q_0$ and for an arbitrary fixed function $f$ satisfying $|f| \leq \mathbf{1}_{Q_0}$.
We define $s$ by $2^{s-1} \leq \ell(Q_0) < 2^s$. As a matter of convenience, we write
$$
\mathcal{T}^{\beta}f(x) := \sum_{\substack{R \in \mathcal{D(\beta)}\\ \ell(R) \leq 2^s}} \mathbf{1}_{R}(x)\int_{\ell(R)/2}^{\ell(R)} \int_\Rn \Big(\frac{t}{t+|x-y|}\Big)^{m\lambda} |\theta_t f(y)|^2 \frac{d\mu(y) dt}{t^{m+1}},
$$
and
$$
\mathscr{F}^\beta:=\bigg(\int_{\Rn} \mathbf{1}_{Q_0}(x) \mathcal{T}^{\beta}f(x)^{p/2} d\mu(x)\bigg)^{1/p}.
$$
Similarly, we also define $\mathcal{T}^{\beta}_{good}$ and $\mathcal{T}^{\beta}_{bad}$ by restricting $R \in \mathcal{D}(\beta)$ on $\mathcal{D}_{good}^\beta$ and $\mathcal{D}_{bad}^\beta$ respectively. Naturally, we have $\mathscr{F}_{good}^\beta$ and $\mathscr{F}_{bad}^\beta$ corresponding to $\mathcal{T}^{\beta}_{good}$ and $\mathcal{T}^{\beta}_{bad}$.
Then we have
\begin{align*}
\int_{0}^{\ell(Q_0)} \int_\Rn \Big(\frac{t}{t+|x-y|}\Big)^{m\lambda} |\theta_t f(y)|^2 \frac{d\mu(y) dt}{t^{m+1}} \leq \mathcal{T}^{\beta}f(x).
\end{align*}
Thus, it is enough to show
\begin{equation}\label{E-F}
\mathbb{E}_{\beta}(\mathscr{F}^\beta)
\leq \big[ C_3 (1+\mathcal{G}) + C_2^{-1} ||g_{\lambda}^*||_{L^p(\mu) \rightarrow L^p(\mu)}/2 \big]\mu(3Q))^{1/p}.
\end{equation}
Actually, it is easy to observe that
$$ \mathbb{E}_{\beta}(\mathscr{F}^\beta)
\leq \mathbb{E}_{\beta}(\mathscr{F}_{good}^\beta) + \mathbb{E}_{\beta}(\mathscr{F}_{bad}^\beta).$$

We first estimate $\mathbb{E}_{\beta}(\mathscr{F}_{bad}^\beta)$. Note that $\mathbb{E}(g^\epsilon) \leq (\mathbb{E} g)^\epsilon$ for each $0 < \epsilon \leq 1$, and that $\mathbb{E}(\mathbf{1}_{bad}(R+\beta)) \leq c(r) \rightarrow 0$ as $r \rightarrow \infty$, see
\cite{H2014, NTV2003}. Another notable fact is that the position and goodness of $R+\beta$ are independent for any $R \in \mathcal{D}_0$. Making use of these facts, we conclude that
\begin{align*}
&\mathbb{E}_{\beta}(\mathscr{F}_{bad}^\beta)
\leq \bigg[\int_{\Rn} \mathbf{1}_{Q_0}(x) \big(\mathbb{E}_{\beta}\mathcal{T}_{bad}^\beta f(x) \big)^{p/2} d\mu(x)\bigg]^{1/p}\\
&=\bigg[\int_{\Rn} \mathbf{1}_{Q_0}(x) \bigg(\sum_{\substack{R \in \mathcal{D}_0\\ \ell(R) \leq 2^s}} \mathbb{E}(\mathbf{1}_{bad}(R+\beta))
\mathbb{E} \Big(\mathbf{1}_{R+\beta}(x)\int_{\ell(R+\beta)/2}^{\ell(R+\beta)} \int_\Rn \Big(\frac{t}{t+|x-y|}\Big)^{m\lambda}\\
&\quad\quad \times |\theta_t f(y)|^2 \frac{d\mu(y) dt}{t^{m+1}} \Big)\bigg)^{p/2} d\mu(x)\bigg]^{1/p}\\
&\leq c(r)^{1/2} ||g_{\lambda}^*(f)||_{L^p(\mu)}
\leq  c(r)^{1/2} ||g_{\lambda}^*||_{L^p(\mu) \rightarrow L^p(\mu)} \mu(3Q_0)^{1/p}.
\end{align*}
For fixed $C_2$, choosing $r$ sufficiently large such that $c(r) \leq (2 C_2)^{-2}$. Hence, it immediately yields that
\begin{align*}
\mathbb{E}_{\beta}(\mathscr{F}_{bad}^\beta)
\leq (2C_2)^{-1} ||g_{\lambda}^*||_{L^p(\mu) \rightarrow L^p(\mu)} \mu(3Q_0)^{1/p}.
\end{align*}
Consequently, the inequality $(\ref{E-F})$ has benn reduced to proving that there exists a constant $C_3$ such that for any  random variable $\beta$, it holds
\begin{equation}\label{F-good}
\mathscr{F}_{good}^\beta \leq C_3 (1+\mathcal{G}) \mu(3Q_0)^{1/p}.
\end{equation}
From now on, $\beta$ is fixed and simply denote $\mathcal{D}=\mathcal{D}(\beta)$, $\mathcal{T}_{good}=\mathcal{T}_{good}^\beta$.
\subsection{Reduction to martingale difference operators.}\label{reduction-martingale}
Since $f$ is supported in $Q_0$, we decompose
$$f=\sum_{\substack{Q^* \in \mathcal{D}\\ \ell(Q^*)=2^s \\ Q^* \cap Q_0 \neq \emptyset}} \sum_{\substack{Q \in \mathcal{D} \\ Q \subset Q^*}}\Delta_Q f.$$
We define the truncation of operator $\mathcal{T}_{good}$
$$ \mathcal{T}_{good,\zeta} f(x) := \sum_{\substack{R \in \mathcal{D}_{good}\\ 2^{-\zeta} < \ell(R) \leq 2^s}} \mathbf{1}_{R}(x)\int_{\ell(R)/2}^{\ell(R)}
\int_\Rn \Big(\frac{t}{t+|x-y|}\Big)^{m\lambda} |\theta_t f(y)|^2 \frac{d\mu(y) dt}{t^{m+1}},\ \zeta >0.$$
Therefore, for every $x \in Q_0$, we have
\begin{align*}
&\Big| \mathcal{T}_{good} f(x) - \mathcal{T}_{good,\zeta} \Big(\sum_{\substack{Q^* \in \mathcal{D}\\ \ell(Q^*)=2^s \\ Q^* \cap Q_0 \neq \emptyset}}
\sum_{\substack{Q \subset Q^* \\ \ell(Q)>2^{-\zeta}}}\Delta_Q f\Big)(x) \Big| \\
&\leq \big| \mathcal{T}_{good} f(x) - \mathcal{T}_{good,\zeta} f(x) \big| + \Big| \mathcal{T}_{good,\zeta} f(x) - \mathcal{T}_{good,\zeta}
\Big(\sum_{\substack{Q^* \in \mathcal{D}\\ \ell(Q^*)=2^s \\ Q^* \cap Q_0 \neq \emptyset}} \sum_{\substack{Q \subset Q^* \\ \ell(Q)>2^{-\zeta}}}\Delta_Q
f\Big)(x)\Big| \\
&\lesssim \sum_{\substack{R \in \mathcal{D}_{good}\\ \ell(R) \leq 2^{-\zeta} }} \mathbf{1}_{R}(x)\int_{\ell(R)/2}^{\ell(R)} \int_\Rn
\Big(\frac{t}{t+|x-y|}\Big)^{m\lambda} |\theta_t f(y)|^2 \frac{d\mu(y) dt}{t^{m+1}} \\
&\quad\quad + \mathcal{T}_{good,\zeta} \Big(f - \sum_{\substack{Q^* \in \mathcal{D}\\ \ell(Q^*)=2^s \\ Q^* \cap Q_0 \neq \emptyset}} \sum_{\substack{Q \subset
Q^* \\ \ell(Q)>2^{-\zeta}}}\Delta_Q f\Big)(x) \\
&\leq \int_{0}^{2^{-\zeta}} \int_\Rn \Big(\frac{t}{t+|x-y|}\Big)^{m\lambda} |\theta_t f(y)|^2 \frac{d\mu(y) dt}{t^{m+1}} + g_{\lambda}^*\Big(f -
\sum_{\substack{Q^* \in \mathcal{D}\\ \ell(Q^*)=2^s \\ Q^* \cap Q_0 \neq \emptyset}} \sum_{\substack{Q \subset Q^* \\ \ell(Q)>2^{-\zeta}}}\Delta_Q f\Big)(x)^2.
\end{align*}
Notice that $g_{\lambda}^*$ is priori bounded. Then Lebesgue's dominated convergence theorem gives that
$$ \lim_{\zeta \rightarrow \infty} \bigg( \int_{\Rn} \mathbf{1}_{Q_0} \Big|\mathcal{T}_{good} f(x) - \mathcal{T}_{good,\zeta} \Big(\sum_{\substack{Q^* \in
\mathcal{D}\\ \ell(Q^*)=2^s \\ Q^* \cap Q_0 \neq \emptyset}} \sum_{\substack{Q \subset Q^* \\ \ell(Q)>2^{-\zeta}}}\Delta_Q f(x) \Big) \Big|^{p/2}
d\mu(x)\bigg)^{1/p} = 0.$$
This gives us that, in order to obtain $(\ref{F-good})$, we only need to prove
\begin{equation}\aligned\label{Q0-mar}
&\bigg[\int_{\Rn} \mathbf{1}_{Q_0}(x)\bigg(\sum_{\substack{R \in \mathcal{D}_{good}\\ 2^{-\zeta} < \ell(R) \leq 2^s}}
\mathbf{1}_{R}(x)\int_{\ell(R)/2}^{\ell(R)} \int_\Rn \Big(\frac{t}{t+|x-y|}\Big)^{m\lambda} \\
&\quad\quad\quad\quad\quad \times \Big| \theta_t \Big( \sum_{\substack{Q^* \in \mathcal{D}\\ \ell(Q^*)=2^s \\ Q^* \cap Q_0 \neq \emptyset}} \sum_{\substack{Q
\in \mathcal{D} \\ Q \subset Q^* \\ \ell(Q)>2^{-\zeta}}}\Delta_Q f \Big)(y) \Big|^2 \frac{d\mu(y) dt}{t^{m+1}}\bigg)^{p/2} d\mu(x)\bigg]^{1/p} \\
&\leq C_3 (1+\mathcal{G}) \mu(3Q_0)^{1/p}.
\endaligned
\end{equation}
It is important to note that there are only finitely many $Q^*$ belonging to the collection
$\{Q^* \in \mathcal{D}; \ell(Q^*)=2^s, Q^* \cap Q_0 \neq \emptyset \}$. And for such $Q^*$, it holds that $Q^* \subset 3Q_0$.
Accordingly, to conclude $(\ref{Q0-mar})$, we only need to show that for each fixed $\zeta$ and for each fixed $Q^*$, there holds that
\begin{equation}\aligned\label{Q*-mar}
&\bigg[\int_{\Rn} \mathbf{1}_{Q_0}(x)\bigg(\sum_{\substack{R \in \mathcal{D}_{good}\\ 2^{-\zeta} < \ell(R) \leq 2^s}}
\mathbf{1}_{R}(x)\int_{\ell(R)/2}^{\ell(R)} \int_\Rn \Big(\frac{t}{t+|x-y|}\Big)^{m\lambda} \\
&\quad\quad\quad\quad\quad \times \Big| \sum_{\substack{Q \in \mathcal{D}: Q \subset Q^* \\ \ell(Q)>2^{-\zeta}}}\theta_t(\Delta_Q f)(y) \Big|^2 \frac{d\mu(y)
dt}{t^{m+1}}\bigg)^{p/2} d\mu(x)\bigg]^{1/p} \\
&\leq C_3 (1+\mathcal{G}) \mu(Q^*)^{1/p}.
\endaligned
\end{equation}
\qed
\section{Proof of the reduction of Carleson estimate $(\ref{key-2})$}
In this section, we undertake to deal with the estimate of the inequality $(\ref{Q*-mar})$. In any case, for fixed cube $R \in \mathcal{D}_{good}$, we perform
the splitting
$$ \sum_{Q \in \mathcal{D}}
=\sum_{\ell(Q)<\ell(R)}
+\sum_{\substack{\ell(Q) \geq \ell(R) \\ d(Q,R) > \ell(R)^{\gamma} \ell(Q)^{1-\gamma}}}
+\sum_{\substack{\ell(R) \leq \ell(Q) \leq 2^r \ell(R) \\ d(Q,R) \leq \ell(R)^{\gamma}\ell(Q)^{1-\gamma}}}
+\sum_{\substack{\ell(Q) > 2^r \ell(R)\\d(Q,R) \leq \ell(R)^{\gamma} \ell(Q)^{1-\gamma}}}. $$
Thus, the left hand side of $(\ref{Q*-mar})$ is dominated by correspondingly four pieces, which are denoted
$\mathcal{G}_{less}$,$\mathcal{G}_{sep}$,$\mathcal{G}_{adj}$ and $\mathcal{G}_{nes}$ . We shall discuss the four terms.

Throughout the paper, the inequality $2^j < D(Q,R)/\ell(R) \leq 2^{j+1}$ will be abbreviated as $D(Q,R)/\ell(R) \thicksim 2^j$, where
$D(Q,R)=\ell(Q)+\ell(R)+d(Q,R)$. Let $\theta(j):=\lceil \frac{j \gamma + r}{1-\gamma}\rceil$, where $\lceil a \rceil$ is the smallest integer bigger than or
equal to $a$. Additionally, we use $Q^{(k)} \in \mathcal{D}$ to denote the unique cube for which $\ell(Q^{(k)}) = 2^k \ell(Q)$ and $Q \subset Q^{(k)}$. The
cube $Q^{(k)}$ is called as the $k$ generation older dyadic ancestor of $Q$. The Whitney region $W_R = R \times (\ell(R)/2, \ell(R)]$ for any $R \in
\mathcal{D}$.

\subsection{The Case $\ell(Q) < \ell(R)$.}\label{sec-less}
Before starting the proof, we first present a key lemma.
\begin{lemma}\label{estimate-1}
Let $0 < \alpha \leq m(\lambda - 2)/2$. Assume that $Q, R \in \mathcal{D}$ are given cubes satisfying $\ell(R)=2^i \ell(Q)$ and $D(Q,R)/\ell(R) \thicksim 2^j $
for $i \geq 1$, $j \geq 0$. Suppose that $\ell(Q) < \ell(R) \leq 2^s$ and $(x,t) \in W_R$. Then, if
$S = Q^{(i+j+\theta(j))}$, we have the following estimate
\begin{equation*}
\bigg( \int_{\Rn} |\theta_t(\Delta_Q f)(y)|^2 \Big(\frac{t}{t+|x-y|}\Big)^{m\lambda}\frac{d\mu(y)}{t^m}\bigg)^{1/2}
\lesssim 2^{-\alpha(i+\frac{3}{4}j)} \ell(S)^{-m} \big\| \Delta_Q f \big\|_{L^1(\mu)}.
\end{equation*}
\end{lemma}

\begin{proof}
We begin by noting that
$$\frac{\ell(Q)^\alpha}{D(Q,R)^{m+\alpha}} \lesssim 2^{-\alpha(i+\frac{3}{4}j)} \ell(S)^{-m}.$$
In fact, the definition of $S$ gives that
$$ \ell(S) = 2^{i+j+\theta(j)} \ell(Q) = 2^{j+\theta(j)} \ell(R) \thickapprox 2^{j+j \frac{\gamma}{1-\gamma}} \ell(R).$$
Thus, we by $\frac{m\gamma}{1-\gamma} < \frac{\alpha}{4}$ have
\begin{align*}
\frac{\ell(Q)^\alpha}{D(Q,R)^{m+\alpha}}
\lesssim 2^{-\alpha i} 2^{-(m+\alpha)j}  \ell(R)^{-m}
\lesssim 2^{-\alpha i} 2^{-(\alpha-\frac{m\gamma}{1-\gamma})j} \ell(S)^{-m}
< 2^{-\alpha(i+\frac{3}{4}j)} \ell(S)^{-m}.
\end{align*}

Let $c_Q$ be the center of $Q$. Since $\ell(Q) < \ell(R) \leq 2^s$, we get the vanishing property $\int_Q \Delta_Q f d\mu =0$. Then we have
$$|\theta_t(\Delta_Q f)(y)|
= \bigg|\int_{Q}(s_t(y,z)-s_t(y,c_Q)) \Delta_Q f(z) d\mu(z)\bigg|
\lesssim \int_{Q} \frac{\ell(Q)^\alpha}{(t+|y-z|)^{m+\alpha}}|\Delta_Q f(z)| d\mu(z). $$
For $x \in R$ and $z \in Q$, $|x-z|\geq d(Q,R)$. We will consider two subcases.

First, we analyze the contribution of the subregion in which $y:|y-x|\leq \frac12 d(Q,R)$.
In this case, $t + |y-z| \geq t + |x-z|-|x-y| \gtrsim \ell(R) + d(Q,R) \thickapprox D(Q,R)$.
Thus, we get
$$|\theta_t(\Delta_Q f)(y)|
\lesssim \frac{\ell(Q)^\alpha}{D(Q,R)^{m+\alpha}} \big\| \Delta_Q f \big\|_{L^1(\mu)},$$
and
\begin{align*}
\bigg( \int_{y:|y-x|\leq \frac12 d(Q,R)} |\theta_t(\Delta_Q f)(y)|^2 \Big(\frac{t}{t+|x-y|}\Big)^{m\lambda}\frac{d\mu(y)}{t^m}\bigg)^{1/2}
\lesssim 2^{-\alpha(i+\frac{3}{4}j)} \ell(S)^{-m} \big\| \Delta_Q f \big\|_{L^1(\mu)}.
\end{align*}

Secondly, we treat the contribution made by those $y:|y-x| > \frac12 d(Q,R)$. Since $|y-x| > \frac12 d(Q,R)$, we have
$$ \frac{t}{t+|x-y|} \lesssim \frac{\ell(R)}{\ell(R)+d(Q,R)}.$$
Accordingly, together with Young's inequality, this yields that
\begin{align*}
&\bigg(\int_{y:|y-x|> \frac12 d(Q,R)} |\theta_t(\Delta_Q f)(y)|^2 \Big(\frac{t}{t+|x-y|}\Big)^{m\lambda}\frac{d\mu(y)}{t^m}\bigg)^{1/2} \\
&\lesssim \frac{\ell(R)^{\frac{m\lambda}{2}-m-\alpha} \ell(Q)^\alpha t^{-m/2}}{(\ell(R)+d(Q,R))^{\frac{m\lambda}{2}}}
\bigg[\int_{\Rn}\bigg(\int_{\Rn}\Big(\frac{t}{(t+|y-z|)}\Big)^{m+\alpha} |\Delta_Q f(z)| d\mu(z)\bigg)^2 d\mu(y)\bigg]^{1/2} \\
&\leq \frac{\ell(Q)^\alpha}{(\ell(R)+d(Q,R))^{m+\alpha}} t^{-m/2}  \Big\|\Big(\frac{t}{t+|\cdot|}\Big)^{m+\alpha}\Big\|_{L^2(\mu)} \big\|\Delta_Q f
\big\|_{L^1(\mu)} \\
&\lesssim \frac{\ell(Q)^\alpha}{D(Q,R)^{m+\alpha}} \big\| \Delta_Q f \big\|_{L^1(\mu)}
\lesssim 2^{-\alpha(i+\frac{3}{4}j)} \ell(S)^{-m} \big\| \Delta_Q f \big\|_{L^1(\mu)},
\end{align*}
where we have used the condition $0 < \alpha \leq m(\lambda - 2)/2$.

The proof of Lemma $\ref{estimate-1}$ is complete.
\end{proof}

Now we begin to bound the first part $\mathcal{G}_{less}$. In this case, the restrictions on $Q,R$ enable us to rewrite the summing index
\begin{align*}
\mathcal{G}_{less}
&=\bigg[\int_{\Rn} \mathbf{1}_{Q_0}(x)\bigg(\sum_{k \leq s} \sum_{\substack{R \in \mathcal{D}_{k,good}\\ 2^{-\zeta} < \ell(R) \leq 2^s}}
\mathbf{1}_{R}(x)\int_{\ell(R)/2}^{\ell(R)} \int_\Rn \Big(\frac{t}{t+|x-y|}\Big)^{m\lambda} \\
&\quad\quad\quad\quad\quad \times \Big| \sum_{i \geq 1}\sum_{j \geq 0}\sum_{\substack{Q \in \mathcal{D}_{k-i}:Q \subset Q^* \\ \ell(Q)>2^{-\zeta} \\
D(Q,R)/\ell(R) \thicksim 2^j}}\theta_t(\Delta_Q f)(y) \Big|^2 \frac{d\mu(y) dt}{t^{m+1}}\bigg)^{p/2} d\mu(x)\bigg]^{1/p}.
\end{align*}
Moreover, notice that it must have $\ell(Q) < 2^s$. Then using the Minkowski's inequality and Lemma $\ref{estimate-1}$ we gain that
\begin{align*}
\mathcal{G}_{less}
&\leq \bigg[\int_{\Rn}\bigg(\sum_{k \leq s} \sum_{\substack{R \in \mathcal{D}_{k,good}\\ \ell(R) \leq 2^s}} \mathbf{1}_{R}(x) \bigg(\sum_{i \geq 1}\sum_{j \geq
0}\sum_{\substack{Q \in \mathcal{D}_{k-i}:Q \subset Q^* \\ D(Q,R)/\ell(R) \thicksim 2^j}} \bigg(\int_{\ell(R)/2}^{\ell(R)} \int_\Rn
\Big(\frac{t}{t+|x-y|}\Big)^{m\lambda} \\
&\quad\quad\quad\quad\quad \times |\theta_t(\Delta_Q f)(y)|^2 \frac{d\mu(y) dt}{t^{m+1}}\bigg)^{1/2}\bigg)^2\bigg)^{p/2} d\mu(x)\bigg]^{1/p}\\
&\lesssim \bigg\|\bigg(\sum_{k \leq s} \sum_{\substack{R \in \mathcal{D}_{k,good}\\ \ell(R) \leq 2^s}} \mathbf{1}_{R} \Big(\sum_{i \geq 1}\sum_{j \geq
0}2^{-\alpha(i+\frac{3}{4}j)} \sum_{\substack{Q \in \mathcal{D}_{k-i}:Q \subset Q^* \\ D(Q,R)/\ell(R) \thicksim 2^j}} \ell(S)^{-m} \big\| \Delta_Q f
\big\|_{L^1(\mu)}\Big)^2  \bigg)^{1/2} \bigg\|_{L^p(\mu)}.
\end{align*}
Additionally, combining H\"{o}lder's inequality and the fact that $(\sum_i |a_i|)^{1/2} \leq \sum_i |a_i|^{1/2}$, we have

\begin{align*}
\mathcal{G}_{less}
&\lesssim \bigg\|\bigg(\sum_{k \leq s} \sum_{R \in \mathcal{D}_{k,good}} \mathbf{1}_{R} \sum_{i \geq 1}\sum_{j \geq 0}2^{-\alpha(i+\frac{3}{4}j)}
\Big(\sum_{\substack{Q \in \mathcal{D}_{k-i}:Q \subset Q^* \\ D(Q,R)/\ell(R) \thicksim 2^j}} \ell(S)^{-m} \big\| \Delta_Q f \big\|_{L^1(\mu)}\Big)^2
\bigg)^{1/2} \bigg\|_{L^p(\mu)} \\
&\leq \bigg\|\sum_{i \geq 1}\sum_{j \geq 0}2^{-\frac{\alpha}{2}(i+\frac{3}{4}j)}\bigg(\sum_{k \leq s} \sum_{R \in \mathcal{D}_{k,good}} \mathbf{1}_{R}
\Big(\sum_{\substack{Q \in \mathcal{D}_{k-i}:Q \subset Q^* \\ D(Q,R)/\ell(R) \thicksim 2^j}}  \ell(S)^{-m} \big\| \Delta_Q f \big\|_{L^1(\mu)}\Big)^2
\bigg)^{1/2}\bigg\|_{L^p(\mu)} \\
&\leq \sum_{i \geq 1}\sum_{j \geq 0}2^{-\frac{\alpha}{2}(i+\frac{3}{4}j)} \bigg\|\bigg(\sum_{k \leq s} \sum_{R \in \mathcal{D}_{k,good}} \mathbf{1}_{R}
\Big(\sum_{\substack{Q \in \mathcal{D}_{k-i}:Q \subset Q^* \\ D(Q,R)/\ell(R) \thicksim 2^j}} \ell(S)^{-m} \big\| \Delta_Q f \big\|_{L^1(\mu)}\Big)^2
\bigg)^{1/2}\bigg\|_{L^p(\mu)}.
\end{align*}
Then in order to dominate the sum over $R$, we follow the strategy used in \cite{MM}. For the for completeness of this article, we give the proof. For fixed
$j,k$, denote $\tau_j(k)=k+j+\theta(j)$. From the fact that $Q,R \subset S$ (see \cite[p.~483]{H2014}), it follows that
\begin{align*}
&\sum_{R \in \mathcal{D}_{k,good}} \mathbf{1}_{R} \Big(\sum_{\substack{Q \in \mathcal{D}_{k-i}:Q \subset Q^* \\ D(Q,R)/\ell(R) \thicksim 2^j}} \ell(S)^{-m}
\big\| \Delta_Q f \big\|_{L^1(\mu)}\Big)^2 \\
&=\Big(\sum_{R \in \mathcal{D}_{k,good}} \mathbf{1}_{R}\sum_{\substack{Q \in \mathcal{D}_{k-i}:Q \subset Q^* \\ D(Q,R)/\ell(R) \thicksim 2^j}} 2^{-m \tau_j(k)}
\big\| \Delta_Q f \big\|_{L^1(\mu)}\Big)^2 \\
&\leq \Big(\sum_{S \in \mathcal{D}_{\tau_j(k)}} \sum_{\substack{R \in \mathcal{D}_{k,good} \\ R \subset S}} \mathbf{1}_{R} \sum_{\substack{Q \in
\mathcal{D}_{k-i}:Q \subset Q^* \\ D(Q,R)/\ell(R) \thicksim 2^j}} 2^{-m \tau_j(k)} \big\| \Delta_Q f \big\|_{L^1(\mu)}\Big)^2 \\
\end{align*}
\begin{align*}
&\lesssim \Big(\sum_{S \in \mathcal{D}_{\tau_j(k)}} \sum_{\substack{R \in \mathcal{D}_{k,good} \\ R \subset S}} \mathbf{1}_{R} \frac{1}{\mu(S)} \int_S \big|
\Delta_{k-i} f \big| d\mu \Big)^2 \\
&\leq \Big(\sum_{S \in \mathcal{D}_{\tau_j(k)}} \frac{\mathbf{1}_S}{\mu(S)} \int_S \big| \Delta_{k-i} f \big| d\mu \Big)^2
=\big(\mathbb{E}_{\tau_j(k)}(| \Delta_{k-i} f |)\big)^2.
\end{align*}
Therefore, by Lemma $\ref{Stein}$ and Lemma $\ref{Delta-k}$ we have that
\begin{align*}
\mathcal{G}_{less}
&\lesssim \sum_{i \geq 1}\sum_{j \geq 0}2^{-\frac{\alpha}{2}(i+\frac{3}{4}j)} \Big\| \Big(\sum_{k \leq s}\big(\mathbb{E}_{\tau_j(k)}(| \Delta_{k-i} f |)\big)^2
\Big)^{1/2} \Big\|_{L^p(\mu)} \\
&\lesssim \sum_{i \geq 1}\sum_{j \geq 0}2^{-\frac{\alpha}{2}(i+\frac{3}{4}j)} \Big\| \Big(\sum_{k}| \Delta_{k} f |^2 \Big)^{1/2} \Big\|_{L^p(\mu)}
\lesssim \mu(Q^*)^{1/p}.
\end{align*}
\qed
\subsection{The Case $\ell(Q) \geq \ell(R)$ and $d(Q,R) > \ell(R)^{\gamma} \ell(Q)^{1-\gamma}$.}\label{sec-sep}
\begin{lemma}\label{estimate-2}
Let $0 < \alpha \leq m(\lambda - 2)/2$. Assume that $Q, R \in \mathcal{D}$ are given cubes satisfying $\ell(Q)=2^i \ell(R)$,
$d(Q,R) > \ell(R)^{\gamma} \ell(Q)^{1-\gamma}$ and $D(Q,R)/\ell(Q) \thicksim 2^j $ for $i \geq 0$, $j \geq 0$. If $(x,t) \in W_R$,
$S=Q^{(j+\theta(i+j))}$, then we have that

\begin{equation*}
\bigg( \int_{\Rn} |\theta_t(\Delta_Q f)(y)|^2 \Big(\frac{t}{t+|x-y|}\Big)^{m\lambda}\frac{d\mu(y)}{t^m}\bigg)^{1/2}
\lesssim 2^{-\frac{\alpha}{4}(i+j)} \ell(S)^{-m} \big\| \Delta_Q f \big\|_{L^1(\mu)}.
\end{equation*}
\end{lemma}

\begin{proof}
We first prove that there holds that in this case
\begin{equation}\label{ell-d-D}
\frac{\ell(R)^\alpha}{(\ell(R)+d(Q,R))^{m+\alpha}}
\lesssim 2^{-\frac{\alpha}{4}(i+j)} \ell(S)^{-m}.
\end{equation}
This will be divided into two steps to show. The first step is to gain
\begin{equation}\label{step-1}
\frac{\ell(R)^\alpha}{(\ell(R)+d(Q,R))^{m+\alpha}}
\lesssim \frac{\ell(Q)^{\alpha/2} \ell(R)^{\alpha/2}}{D(Q,R)^{m+\alpha}}.
\end{equation}
Actually, if $\ell(Q) \leq d(Q,R)$, it is obvious that
\begin{align*}
\frac{\ell(R)^\alpha}{(\ell(R)+d(Q,R))^{m+\alpha}}
\lesssim \frac{\ell(R)^\alpha}{D(Q,R)^{m+\alpha}}
\leq \frac{\ell(Q)^{\alpha/2} \ell(R)^{\alpha/2}}{D(Q,R)^{m+\alpha}}.
\end{align*}
If $\ell(Q) > d(Q,R)$, then $D(Q,R) \thickapprox \ell(Q)$. Using $d(Q,R) > \ell(R)^{\gamma} \ell(Q)^{1-\gamma}$ and
$\gamma \leq \frac{\alpha}{2(m+\alpha)}$, we obtain
\begin{align*}
\ell(Q) = \bigg(\frac{\ell(Q)}{\ell(R)}\bigg)^\gamma \ell(R)^\gamma \ell(Q)^{1-\gamma}
< \bigg(\frac{\ell(Q)}{\ell(R)}\bigg)^\gamma d(Q,R),
\end{align*}
and
\begin{align*}
\frac{\ell(R)^\alpha}{(\ell(R)+d(Q,R))^{m+\alpha}}
\leq \frac{\ell(R)^\alpha}{d(Q,R)^{m+\alpha}}
\leq \frac{\ell(Q)^{\alpha/2} \ell(R)^{\alpha/2}}{\ell(Q)^{m+\alpha}}
\thickapprox \frac{\ell(Q)^{\alpha/2} \ell(R)^{\alpha/2}}{D(Q,R)^{m+\alpha}}.
\end{align*}
The second step is to observe that
\begin{equation}\label{step-2}
\frac{\ell(Q)^{\alpha/2} \ell(R)^{\alpha/2}}{D(Q,R)^{m+\alpha}} \lesssim 2^{-\frac{\alpha}{4}(i+j)} \ell(S)^{-m}.
\end{equation}
Indeed, using the definition of $S_0$, we have
$$ \ell(S_0) = 2^{j+\theta(i+j)} \ell(Q) \thickapprox 2^{j+(i+j) \frac{\gamma}{1-\gamma}} \ell(Q).$$
Hence, it follows from $\frac{m\gamma}{1-\gamma} < \frac{\alpha}{4}$ that
\begin{align*}
\frac{\ell(Q)^{\alpha/2} \ell(R)^{\alpha/2}}{D(Q,R)^{m+\alpha}}
\thickapprox 2^{-\frac{\alpha}{2} i} 2^{-(m+\alpha)j}  \ell(Q)^{-m}
\lesssim 2^{-\alpha i} 2^{-\alpha j + \frac{m\gamma}{1-\gamma}(i+j)} \ell(S_0)^{-m}
< 2^{-\frac{\alpha}{4}(i+j)} \ell(S_0)^{-m}.
\end{align*}
Combining the inequalities $(\ref{step-1})$ and $(\ref{step-2})$ gives the desired result $(\ref{ell-d-D})$.

Next, we continue with the proof. Using the size condition, we get
$$|\theta_t(\Delta_Q f)(y)|
\lesssim \int_{\Rn}\frac{t^\alpha}{(t+|y-z|)^{m+\alpha}} |\Delta_Q f(z)|d\mu(z). $$
Applying the similar argument as in Lemma $\ref{estimate-1}$, we have the following estimate
\begin{equation*}
\bigg( \int_{\Rn} |\theta_t(\Delta_Q f)(y)|^2 \Big(\frac{t}{t+|x-y|}\Big)^{m\lambda}\frac{d\mu(y)}{t^m}\bigg)^{1/2}
\lesssim \ 2^{-\frac{\alpha}{4}(i+j)} \ell(S)^{-m} \big\| \Delta_Q f \big\|_{L^1(\mu)}.
\end{equation*}
This shows Lemma $\ref{estimate-2}$.
\end{proof}

Now we turn to the estimate of separated part. We reindex the sum over $R$.
\begin{align*}
\mathcal{G}_{sep}
&=\bigg[\int_{\Rn} \mathbf{1}_{Q_0}(x)\bigg(\sum_{k \leq s} \sum_{\substack{R \in \mathcal{D}_{k,good}\\ 2^{-\zeta} < \ell(R) \leq 2^s}}
\mathbf{1}_{R}(x)\int_{\ell(R)/2}^{\ell(R)} \int_\Rn \Big(\frac{t}{t+|x-y|}\Big)^{m\lambda} \\
&\quad\quad\quad\quad\quad \times \Big| \sum_{i \geq 0}\sum_{j \geq 0}\sum_{\substack{Q \in \mathcal{D}_{k+i}:Q \subset Q^* \\ \ell(Q)>2^{-\zeta} \\
D(Q,R)/\ell(R) \thicksim 2^j \\ d(Q,R) > \ell(R)^{\gamma} \ell(Q)^{1-\gamma}}}\theta_t(\Delta_Q f)(y) \Big|^2 \frac{d\mu(y) dt}{t^{m+1}}\bigg)^{p/2}
d\mu(x)\bigg]^{1/p}.
\end{align*}
Proceeding as we did in the previous subsection, we obtain that
$\mathcal{G}_{sep} \lesssim \mu(Q^*)^{1/p}$.
\subsection{The Case $\ell(R) \leq \ell(Q) \leq 2^r \ell(R)$ and $d(Q,R) \leq \ell(R)^{\gamma} \ell(Q)^{1-\gamma}$.}\label{sec-adj}
In this case, if we define $i,j$ by $\ell(Q) = 2^i \ell(R)$ and $D(Q,R)/\ell(Q) \thicksim 2^j$, then it is trivial that $0 \leq i \leq r$ and $0 \leq j \leq
1$. Moreover, notice that
$$ \frac{\ell(R)^\alpha}{(\ell(R)+d(Q,R))^{m+\alpha}} \leq \ell(R)^{-m} \thickapprox \frac{\ell(Q)^{\alpha/2} \ell(R)^{\alpha/2}}{D(Q,R)^{m+\alpha}} .$$
This  parallels with $(\ref{step-1})$.
A completely analogous calculation to that of the preceding section yields that
$\mathcal{G}_{adj} \lesssim \mu(Q^*)^{1/p}$.
\subsection{The Case $\ell(Q) > 2^r \ell(R)$ and $d(Q,R) \leq \ell(R)^{\gamma} \ell(Q)^{1-\gamma}$.}
In this case, since $R$ is good, it must actually have $R \subset Q$. That is, $Q$ is the ancestor of $R$. Then we can write
\begin{align*}
\mathcal{G}_{nes}
= \bigg\| \mathbf{1}_{Q_0}\bigg(& \sum_{\substack{R \in \mathcal{D}_{good}:R \subset Q^*\\ 2^{-\zeta} < \ell(R) \leq 2^{s-r-1}}}
\mathbf{1}_{R}\int_{\ell(R)/2}^{\ell(R)} \int_\Rn \Big(\frac{t}{t+|x-y|}\Big)^{m\lambda} \\
&\quad \times \Big| \sum_{i=r+1}^{s-\log_2 \ell(R)} \theta_t(\Delta_{R^{(i)}} f)(y)
\Big|^2 \frac{d\mu(y) dt}{t^{m+1}}\bigg)^{1/2}\bigg\|_{L^p(\mu)}.
\end{align*}
If we by $\mathcal{G}_{nes}'$ and $\mathcal{G}_{nes}^{''}$ denote the sum with respect to the terms $\mathbf{1}_{R^{(i)} \setminus R^{(i-1)}}\Delta_{R^{(i)}}
f$ and $\mathbf{1}_{R^{(i-1)}}\Delta_{R^{(i)}} f$, there holds
$$\mathcal{G}_{nes} \leq \mathcal{G}_{nes}' + \mathcal{G}_{nes}^{''}.$$

\subsubsection{\textbf{Bound for} $\mathcal{G}_{nes}'$.}
One begins by dominating $\mathcal{G}_{nes}'$ with
\begin{align*}
\mathcal{G}_{nes}'
\leq \bigg\| \mathbf{1}_{Q_0}\bigg(\sum_{k \leq s} & \sum_{\substack{R \in \mathcal{D}_{k,good}:R \subset Q^*\\ 2^{-\zeta} < \ell(R) \leq 2^{s-r-1}}}
\mathbf{1}_{R}\bigg(\sum_{i=r+1}^{s-\log_2 \ell(R)} \Big(\int_{\ell(R)/2}^{\ell(R)} \int_\Rn \Big(\frac{t}{t+|x-y|}\Big)^{m\lambda}\\ &\quad\quad\quad \times
\Big| \theta_t(\mathbf{1}_{R^{(i)} \setminus R^{(i-1)}}\Delta_{R^{(i)}} f)(y) \Big|^2 \frac{d\mu(y) dt}{t^{m+1}}\Big)^{1/2}\bigg)^2
\bigg)^{1/2}\bigg\|_{L^p(\mu)} \\
\end{align*}
Now we need the following lemma.
\begin{lemma}\label{nested-1}
Let $i \geq r+1$ and $R \in \mathcal{D}_{k,good}$. Then we have for $(x,t) \in W_R$ that
\begin{equation*}
\bigg(\int_{\Rn} |\theta_t(\mathbf{1}_{R^{(i)} \setminus R^{(i-1)}} \Delta_{R^{(i)}} f)(y)|^2  \Big(\frac{t}{t+|x-y|}\Big)^{n\lambda} \frac{
d\mu(y)}{t^{m}}\bigg)^{1/2}
\lesssim 2^{-\frac{\alpha}{2} i} 2^{-(k+i)m} \big\| \Delta_{R^{(i)}} f \big\|_{L^1(\mu)}.
\end{equation*}
\end{lemma}
Using the above result we see that
\begin{align*}
\mathcal{G}_{nes}'
&\lesssim \bigg\| \mathbf{1}_{Q_0}\bigg(\sum_{k \leq s-r-1}\sum_{R \in \mathcal{D}_{k,good}:R \subset Q^*} \mathbf{1}_{R}\Big(\sum_{i=r+1}^{s-\log_2 \ell(R)}
2^{-\frac{\alpha}{2}i}2^{-(k+i)m} || \Delta_{R^{(i)}} f ||_{L^1(\mu)}\Big)^2 \bigg)^{1/2}\bigg\|_{L^p(\mu)} \\
&\leq \bigg\| \mathbf{1}_{Q_0}\bigg(\sum_{k \leq s-r-1}\Big(\sum_{i=r+1}^{s-\log_2 \ell(R)}  2^{-\frac{\alpha}{2}i} \sum_{R \in \mathcal{D}_{k,good}:R \subset
Q^*} \mathbf{1}_{R} 2^{-(k+i)m} || \Delta_{R^{(i)}} f ||_{L^1(\mu)}\Big)^2 \bigg)^{1/2}\bigg\|_{L^p(\mu)},
\end{align*}
where we used $\sum_k |a_k|^2 \leq (\sum_k |a_k|)^2$. Then discrete Minkowski's inequality yields that
\begin{align*}
\mathcal{G}_{nes}'
&\lesssim \bigg\| \mathbf{1}_{Q_0}\sum_{i=r+1}^{s-\log_2 \ell(R)}  2^{-\frac{\alpha}{2}i} \bigg(\sum_{k \leq s-i}\Big( \sum_{R \in \mathcal{D}_{k,good}:R
\subset Q^*} \mathbf{1}_{R} 2^{-(k+i)m} || \Delta_{R^{(i)}} f ||_{L^1(\mu)}\Big)^2 \bigg)^{1/2}\bigg\|_{L^p(\mu)} \\
&\leq \sum_{i=r+1}^{s-\log_2 \ell(R)}  2^{-\frac{\alpha}{2}i} \bigg\| \mathbf{1}_{Q_0} \bigg(\sum_{k \leq s-i}\Big( \sum_{R \in \mathcal{D}_{k,good}:R \subset
Q^*} \mathbf{1}_{R} 2^{-(k+i)m} || \Delta_{R^{(i)}} f ||_{L^1(\mu)}\Big)^2 \bigg)^{1/2}\bigg\|_{L^p(\mu)}.
\end{align*}
Finally, making using of Lemma $\ref{Stein}$ and Lemma $\ref{Delta-k}$ again, we get that
\begin{equation}\aligned\label{G-nes'}
\mathcal{G}_{nes}'
&\leq \sum_{i=r+1}^{s-\log_2 \ell(R)}  2^{-\frac{\alpha}{2}i} \bigg\| \mathbf{1}_{Q_0} \bigg(\sum_{k \leq s-i}\Big( \sum_{S \in \mathcal{D}_{k+i}:S \subset
Q^*} \frac{\mathbf{1}_{S}}{\mu(S)} || \Delta_{S} f ||_{L^1(\mu)}\Big)^2 \bigg)^{1/2}\bigg\|_{L^p(\mu)} \\
&= \sum_{i=r+1}^{s-\log_2 \ell(R)}  2^{-\frac{\alpha}{2}i} \bigg\| \mathbf{1}_{Q_0} \bigg(\sum_{k \leq s-i}\Big( \sum_{S \in \mathcal{D}_{k+i}}
\frac{\mathbf{1}_{S}}{\mu(S)} \int_S |\Delta_{k+i}f | d\mu)\Big)^2 \bigg)^{1/2}\bigg\|_{L^p(\mu)} \\
&\lesssim  \sum_{i=r+1}^{s-\log_2 \ell(R)}  2^{-\frac{\alpha}{2}i} \Big\| \Big(\sum_{k \leq s}\big(\mathbb{E}_{k}(| \Delta_{k} f |)\big)^2 \Big)^{1/2}
\Big\|_{L^p(\mu)} \\
&\lesssim  \sum_{i=r+1}^{s-\log_2 \ell(R)}  2^{-\frac{\alpha}{2}i} \Big\| \Big(\sum_{k}| \Delta_{k} f |^2 \Big)^{1/2} \Big\|_{L^p(\mu)}
\lesssim \mu(Q^*)^{1/p}.
\endaligned
\end{equation}
So, we are left to prove Lemma $\ref{nested-1}$.

\vspace{0.3cm}
\noindent\textbf{Proof of Lemma $\ref{nested-1}$.}
First, we shall prove, for any $z \in (R^{(i-1)})^c$,
\begin{equation}\label{goodness}
\frac{\ell(R)^\alpha}{(\ell(R)+\dist(z,R))^{m+\alpha}} \lesssim 2^{-\frac{\alpha}{2} i} 2^{-(k+i)m} .
\end{equation}
On the one hand, making use of the fact $i \geq r+1$ and the goodness of $R$ , we have
$$ d(R,(R^{(i-1)})^c) > \ell(R)^{\gamma} \ell(R^{(i-1)})^{1-\gamma}.$$
On the other hand, since $\dist(z,R) \geq \dist(z,R^{(i-1)}) + \dist(R,\partial R^{(i-1)})$, we obtain that
\begin{align*}
\frac{\ell(R)^\alpha}{(\ell(R)+\dist(z,R))^{n+\alpha}}
&= \bigg(\frac{\ell(R)}{\ell(R^{(i-1)})}\bigg)^\alpha \frac{\ell(R^{(i-1)})^\alpha}{(\ell(R) + \dist(z,R))^{m+\alpha}} \\
&\lesssim \bigg( \frac{\ell(R)}{\ell(R^{(i-1)})} \bigg)^{\alpha -(m+\alpha)\gamma} \frac{\ell(R^{(i-1)})^\alpha}{(\ell(R^{(i-1)})+
\dist(z,R^{(i-1)}))^{m+\alpha}}\\
&\lesssim \bigg( \frac{\ell(R)}{\ell(R^{(i-1)})} \bigg)^{\alpha/2} \ell(R^{(i-1)})^{-m}.
\end{align*}
The above argument yields $(\ref{goodness})$.

Secondly, we by $\mathcal{L}$ denote the left hand side of the inequality in Lemma $\ref{nested-1}$. It follows form the size condition that
\begin{align*}
\mathcal{L} & \lesssim \bigg[ \int_{\Rn} \bigg(\int_{R^{(i-1)^c}} \frac{t^\alpha}{(t + |y - z|)^{m + \alpha}} |\Delta_{R^{(i)}}f(z)|d\mu(z) \bigg)^2
\Big(\frac{t}{t + |x-y|}\Big)^{m \lambda_1} \frac{d\mu(y)}{t^m} \bigg]^{1/2} \\
&\leq \bigg[ \int_{\Rn} \bigg(\int_{E_1} \frac{t^\alpha}{(t + |y - z|)^{m + \alpha}} |\Delta_{R^{(i)}}f(z)| d\mu(z) \bigg)^2 \Big(\frac{t}{t + |x-y|}\Big)^{m
\lambda_1} \frac{d\mu(y)}{t^m} \bigg]^{1/2} \\
&\quad + \bigg[ \int_{\Rn} \bigg(\int_{E_2} \frac{t^\alpha}{(t + |y - z|)^{m + \alpha}} |\Delta_{R^{(i)}}f(z)| d\mu(z) \bigg)^2 \Big(\frac{t}{t +
|x-y|}\Big)^{m \lambda_1} \frac{d\mu(y)}{t^m} \bigg]^{1/2} \\
&:= \mathcal{L}_1 + \mathcal{L}_2 ,
\end{align*}
where
\begin{eqnarray*}
   E_1 &=& \big\{z \in (R^{(i-1)})^c; \ 2|x - y| \leq \dist(z,R) \big\}, \\\
   E_2 &=& \big\{z \in (R^{(i-1)})^c; \ 2|x - y| > \dist(z,R) \big\}.
\end{eqnarray*}

If $z \in E_1$, $$ t + |y - z| \geq t + |x - z| - |x - y| \geq \frac12 (\ell(R) + \dist(z,R)).$$
Hence, together with $(\ref{goodness})$, this leads that
\begin{align*}
\mathcal{L}_1
&\lesssim \int_{E_1} \frac{\ell(R)^\alpha}{(\ell(R)+\dist(z,R))^{n+\alpha}} |\Delta_{R^{(i)}}f(z)| d\mu(z) \bigg(\int_{\Rn}
\Big(\frac{t}{t+|x-y|}\Big)^{m\lambda} \frac{d\mu(y)}{t^{m}}\bigg)^{1/2} \\
&\lesssim 2^{-\frac{\alpha}{2} i} 2^{-(k+i)m} \int_{(R^{(i-1)})^c}|\Delta_{R^{(i)}}f(z)| d\mu(z) .
\end{align*}

If $z \in E_2$, the inequality $(\ref{goodness})$ and Young's inequality imply that

\begin{align*}
\mathcal{L}_2
&\lesssim \bigg[ \int_{\Rn} \bigg(\int_{E_2} \Big(\frac{t}{t + |y - z|}\Big)^{m + \alpha}
\frac{\ell(R)^{\frac{m\lambda}{2}-m}}{(\ell(R)+\dist(z,R))^{\frac{m\lambda}{2}}} |\Delta_{R^{(i)}}f(z)| d\mu(z) \bigg)^2 \frac{d\mu(y)}{t^m} \bigg]^{1/2}\\
&\leq \bigg[ \int_{\Rn} \bigg(\int_{E_2} \Big(\frac{t}{t + |y - z|}\Big)^{m + \alpha} \frac{\ell(R)^{\alpha}}{(\ell(R)+\dist(z,R))^{m+\alpha}}
|\Delta_{R^{(i)}}f(z)| d\mu(z) \bigg)^2 \frac{d\mu(y)}{t^m} \bigg]^{1/2}\\
&\lesssim \bigg[ \int_{\Rn} \bigg(\int_{E_2} \Big(\frac{t}{t + |y - z|}\Big)^{m + \alpha} 2^{-\frac{\alpha}{2} i} 2^{-(k+i)m} |\Delta_{R^{(i)}}f(z)| d\mu(z)
\bigg)^2 \frac{d\mu(y)}{t^m} \bigg]^{1/2}\\
&\lesssim 2^{-\frac{\alpha}{2} i} 2^{-(k+i)m} \int_{(R^{(i-1)})^c}|\Delta_{R^{(i)}}f(z)| d\mu(z) .
\end{align*}
Consequently, Lemma $\ref{nested-1}$ is concluded from the above estimates.
\qed
\subsubsection{\textbf{Bound for} $\mathcal{G}_{nes}^{''}$.}
If $R^{(i-1)}$ is a stopping cube, then we get $(R^{(i-1)})^a=R^{(i-1)}$ and the decomposition
\begin{align*}
\mathbf{1}_{R^{(i-1)}} \Delta_{R^{(i)}}f
=\mathbf{1}_{(R^{(i-1)})^c} \frac{\langle f \rangle_{R^{(i)}}}{\langle b_{(R^{(i)})^a} \rangle_{R^{(i)}}} b_{(R^{(i)})^a}
+ \Big(\frac{\langle f \rangle_{R^{(i-1)}}}{\langle b_{R^{(i-1)}}\rangle_{R^{(i-1)}}} b_{R^{(i-1)}} -
\frac{\langle f \rangle_{R^{(i)}}}{\langle b_{(R^{(i)})^a} \rangle_{R^{(i)}}} b_{(R^{(i)})^a}\Big).
\end{align*}
If $R^{(i-1)}$ is not a stopping cube, then there holds that $(R^{(i-1)})^a=(R^{(i)})^a$ and
$$ \mathbf{1}_{R^{(i-1)}} \Delta_{R^{(i)}}f
= - \mathbf{1}_{(R^{(i-1)})^c} B_{R^{(i-1)}} b_{(R^{(i)})^a} + B_{R^{(i-1)}} b_{(R^{(i)})^a},$$
where
$$ B_{R^{(i-1)}}
=\begin{cases}
\frac{\langle f \rangle_{R^{(i-1)}}}{\langle b_{(R^{(i-1)})^a}\rangle_{R^{(i-1)}}}
 - \frac{\langle f \rangle_{R^{(i)}}}{\langle b_{(R^{(i)})^a} \rangle_{R^{(i)}}}, \ & \ell(R^{(i)}) < 2^s, \\
\frac{\langle f \rangle_{R^{(i-1)}}}{\langle b_{(R^{(i-1)})^a}\rangle_{R^{(i-1)}}} ,   \ & \ell(R^{(i)})=2^s.
\end{cases}
$$
Collecting the two aspects, we can make the following decomposition
\begin{align*}
\sum_{i=r+1}^{s-\log_2 \ell(R)} \mathbf{1}_{R^{(i-1)}}\Delta_{R^{(i)}}f
&= \frac{\langle f \rangle_{R^{(r)}}}{\langle b_{(R^{(r)})^a} \rangle_{R^{(r)}}} b_{(R^{(r)})^a}
- \sum_{i=r+1}^{s-\log_2 \ell(R)} \sum_{R:R^{(i-1)} \notin \mathcal{F}_{Q^*}} \mathbf{1}_{(R^{(i-1)})^c}B_{R^{(i-1)}} b_{(R^{(i)})^a} \\
&\quad + \sum_{i=r+1}^{s-\log_2 \ell(R)} \sum_{R:R^{(i-1)} \in\mathcal{F}_{Q^*}} \mathbf{1}_{(R^{(i-1)})^c} \frac{\langle f
\rangle_{R^{(i)}}}{\langle b_{(R^{(i)})^a} \rangle_{R^{(i)}}} b_{(R^{(i)})^a} .
\end{align*}
Thus, if we set $\mathcal{G}_{nes,par}^{''}$, $\mathcal{G}_{nes,out}^{''}$ and $\mathcal{G}_{nes,in}^{''}$ to be the three terms of the left hand side of the
above equality correspondingly, then
$$ \mathcal{G}_{nes}^{''} \lesssim \mathcal{G}_{nes,par}^{''} + \mathcal{G}_{nes,out}^{''} + \mathcal{G}_{nes,in}^{''} .$$
\vspace{0.3cm}
\noindent\textbf{$\bullet$ The paraproduct estimate.}
By the condition $(2)$ in Theorem $\ref{Theorem}$, the stopping time condition $(1)$ and the fact $|f| \leq 1$, we have
$$ \bigg|\frac{\langle f \rangle_{R^{(r)}}}{\langle b_{(R^{(r)})^a} \rangle_{R^{(r)}}}\bigg| \lesssim 1.$$
Then $\mathcal{G}_{nes,par}^{''}$ can be controlled by
\begin{equation*}\aligned
{}&\bigg\| \mathbf{1}_{Q_0}\bigg(\sum_{\substack{R \in \mathcal{D}_{good}:R \subset Q^* \\ 2^{\zeta} < \ell(R) \leq 2^{s-r-1}}}
\mathbf{1}_{R}\int_{\ell(R)/2}^{\ell(R)} \int_\Rn \Big(\frac{t}{t+|x-y|}\Big)^{m\lambda} |\theta_t(b_{(R^{(r)})^a})(y)|^2 \frac{d\mu(y) dt}{t^{m+1}}
\bigg)^{1/2}\bigg\|_{L^p(\mu)} \\
&\leq \bigg\| \mathbf{1}_{Q_0}\bigg(\sum_{\substack{S \in \mathcal{D} \\ S \subset Q^*}}\sum_{\substack{R \in \mathcal{D}_{good} \\ R^{(r)}=S}}
\mathbf{1}_{R}\int_{\ell(R)/2}^{\ell(R)} \int_\Rn \Big(\frac{t}{t+|x-y|}\Big)^{m\lambda} | \theta_t(b_{S^a})(y) |^2 \frac{d\mu(y) dt}{t^{m+1}}
\bigg)^{1/2}\bigg\|_{L^p(\mu)}.
\endaligned
\end{equation*}
We are going to continue to reindex the sum above in the following way.
\begin{align*}
{}&\bigg\| \mathbf{1}_{Q_0}\bigg(\sum_{F \in \mathcal{F}_{Q^*}}\sum_{S:S^a=F}\sum_{\substack{R \in \mathcal{D} \\ R^{(r)}=S}}
\mathbf{1}_{R}\int_{\ell(R)/2}^{\ell(R)} \int_\Rn \Big(\frac{t}{t+|x-y|}\Big)^{m\lambda} | \theta_t(b_F)(y) |^2 \frac{d\mu(y) dt}{t^{m+1}}
\bigg)^{1/2}\bigg\|_{L^p(\mu)} \\
&\leq \bigg\| \mathbf{1}_{Q_0}\bigg(\sum_{F \in \mathcal{F}_{Q^*}}\sum_{\substack{R \in \mathcal{D} \\ R \subset F}} \mathbf{1}_{R}\int_{\ell(R)/2}^{\ell(R)}
\int_\Rn \Big(\frac{t}{t+|x-y|}\Big)^{m\lambda} | \theta_t(b_F)(y) |^2 \frac{d\mu(y) dt}{t^{m+1}} \bigg)^{1/2}\bigg\|_{L^p(\mu)} \\
&\leq \bigg\| \mathbf{1}_{Q_0}\bigg(\sum_{F \in \mathcal{F}_{Q^*}} \mathbf{1}_{F}\int_{0}^{\ell(F)} \int_\Rn \Big(\frac{t}{t+|x-y|}\Big)^{m\lambda} |
\theta_t(b_F)(y) |^2 \frac{d\mu(y) dt}{t^{m+1}} \bigg)^{1/2}\bigg\|_{L^p(\mu)} .
\end{align*}
Here we need a fundamental inequality. That is, for any nonnegative function sequence $\{ a_{ij}\}$, it holds that
$||(\sum_i \sum_j a_{ij})^{1/2}||_{L^p} \leq \sum_{j}(\sum_{i}||a_{ij}^{1/2}||_{L^p})^{1/p}$,
$1 \leq p \leq 2$.
This gives us that
\begin{align*}
\mathcal{G}_{nes,par}^{''}
&\lesssim \sum_{j \geq 0} \bigg(\sum_{F \in \mathcal{F}^j_{Q^*}} \bigg\| \Big( \mathbf{1}_{F}\int_{0}^{\ell(F)} \int_\Rn \Big(\frac{t}{t+|x-y|}\Big)^{m\lambda}
|\theta_t(b_F)(y)|^2 \frac{d\mu(y) dt}{t^{m+1}} \Big)^{1/2}\bigg\|_{L^p(\mu)}^p \bigg)^{1/p} .
\end{align*}
By testing condition and Lemma $\ref{F-Fj-Q*}$, we obtain
\begin{align*}
\mathcal{G}_{nes,par}^{''}
\leq \mathcal{G} \sum_{j \geq 0} \Big( \sum_{F \in \mathcal{F}^j_{Q^*}} \mu(F) \Big)^{1/p}
\lesssim \mathcal{G} \mu(Q^*)^{1/p}.
\end{align*}
\vspace{0.3cm}
\noindent\textbf{$\bullet$ The non-stopping term.}
In this case, it is important to gain a geometric decay in $i$. Moreover, the goodness is of the essence.

\begin{lemma}\label{nested-2}
Let $0 < \alpha \leq m(\lambda - 2)/2$. Given a cube $R \in \mathcal{D}_{good}$ and $i \geq r+1$, there holds for any $(x,t) \in W_R$
\begin{equation*}
\mathcal{N}(x) := \bigg(\int_{\Rn} |\theta_t(\mathbf{1}_{(R^{(i-1)})^c} b_{(R^{(i)})^a})(y)|^2 \Big(\frac{t}{t+|x-y|}\Big)^{m\lambda}
\frac{d\mu(y)}{t^m} \bigg)^{1/2}
\lesssim 2^{- \alpha i/2}.
\end{equation*}
\end{lemma}

\begin{proof}
Set $N_0$ to satisfy $(R^{(i)})^a = R^{(i+N_0)}$. Since $i > r$, we by the goodness of $R$ have that for $0 \leq j \leq N_0$
\begin{align*}
&d(R, \partial R^{(i+j-1)})^{m+\alpha}
> 2^{(i+j-1)(1-\gamma)(m+\alpha)} \ell(R)^{m+\alpha} \\
&\gtrsim 2^{(i+j)(m+\alpha)/2} \ell(R)^{m+\alpha}
\gtrsim 2^{(i+j)\alpha/2}  \ell(R)^{\alpha} \mu(R^{(i+j)}).
\end{align*}
Thus combining with the stopping conditions we see that
\begin{align*}
\int_{(R^{(i-1)})^c} \frac{\ell(R)^\alpha}{|z - x|^{m+\alpha}} |b_{(R^{(i)})^a}| d\mu(z)
&= \sum_{j=0}^{N_0} \int_{R^{(i+j)} \setminus R^{(i+j-1)}} \frac{\ell(R)^\alpha}{|z - x|^{m+\alpha}} |b_{(R^{(i)})^a}| d\mu(z) \\
&\lesssim \sum_{j=0}^{N_0} \frac{\ell(R)^\alpha \mu(R^{(i+j)})}{ 2^{(i+j)\alpha/2} \mu(R^{(i+j)})} \lesssim 2^{-\alpha i/2}.
\end{align*}
We have used the estimate $ \int_{R^{(i+j)}} |b_{(R^{(i)})^a}| d\mu \lesssim \mu(R^{(i+j)})$. Actually, we only need to note that
for any $j < N_0$, $R^{(i+j)}$ is a non-stopping cube. Thus, the stopping time condition $(2)$ gives the desired result. For $j=N_0$, the condition $(3)$ in
Theorem $\ref{Theorem}$ will be used to get that. Together with the size condition, we split
\begin{align*}
\mathcal{N}(x) & \lesssim \bigg[ \int_{\Rn} \bigg(\int_{(R^{(i)})^a \setminus R^{(i-1)}} \frac{t^\alpha}{(t + |y - z|)^{m + \alpha}} |b_{(R^{(i)})^a}| d\mu(z)
\bigg)^2 \Big(\frac{t}{t + |x-y|}\Big)^{m \lambda_1} \frac{d\mu(y)}{t^m} \bigg]^{1/2} \\
&\leq \bigg[ \int_{\Rn} \bigg(\int_{E_1} \frac{t^\alpha}{(t + |y - z|)^{m + \alpha}} |b_{(R^{(i)})^a}| d\mu(z) \bigg)^2 \Big(\frac{t}{t + |x-y|}\Big)^{m
\lambda_1} \frac{d\mu(y)}{t^m} \bigg]^{1/2} \\
&\quad + \bigg[ \int_{\Rn} \bigg(\int_{E_2} \frac{t^\alpha}{(t + |y - z|)^{m + \alpha}} |b_{(R^{(i)})^a}| d\mu(z) \bigg)^2 \Big(\frac{t}{t + |x-y|}\Big)^{m
\lambda_1} \frac{d\mu(y)}{t^m} \bigg]^{1/2} \\
&:= \mathcal{N}_1(x) + \mathcal{N}_2(x) ,
\end{align*}
where
$$ E_1 = \big\{z \in (R^{(i)})^a \setminus R^{(i-1)}; |z - x| \geq 2|x - y| \big\}, \
   E_2 = \big\{z \in (R^{(i)})^a \setminus R^{(i-1)}; |z - x| < 2|x - y| \big\}.$$

When $z \in E_1$, $$ t + |y - z| > |x - z| - |x - y| \geq \frac12 |x - z|.$$
So, we obtain
$$ \mathcal{N}_1(x)
\lesssim \int_{(R^{(k-1)})^c} \frac{\ell(R)^\alpha}{|z - x|^{m+\alpha}} |b_{(R^{(i)})^a}| d\mu(z) \cdot
\bigg[ \int_{\Rn} \Big(\frac{t}{t + |x-y|}\Big)^{m \lambda} \frac{d\mu(y)}{t^m} \bigg]^{1/2}
\lesssim 2^{-\alpha i/2}.$$
As for $\mathcal{N}_2(x) $, by Young's inequality we have the following estimate
\begin{align*}
\mathcal{N}_2(x)
&\lesssim \bigg[ \int_{\Rn} \bigg(\int_{E_2} \Big(\frac{t}{(t+|y-z|)}\Big)^{m+\alpha} \frac{t^{\frac{m \lambda}{2}-m}}{(t + |x - z|)^{\frac{m \alpha}{2}}}
|b_{(R^{(i)})^a}| d\mu(z) \bigg)^2 \frac{d\mu(y)}{t^m} \bigg]^{1/2} \\
&\lesssim \bigg[ \int_{\Rn} \bigg(\int_{E_2} \Big(\frac{t}{(t+|y-z|)}\Big)^{m+\alpha} \frac{t^\alpha}{(t + |x - z|)^{m+\alpha}} |b_{(R^{(i)})^a}| d\mu(z)
\bigg)^2 \frac{d\mu(y)}{t^m} \bigg]^{1/2} \\
&=t^{-m/2} \big\| \xi * \eta \big\|_{L^2(\mu)}
\leq t^{-m/2} \big\| \xi \big\|_{L^2(\mu)} \big\| \eta \big\|_{L^1(\mu)} \\
&\lesssim \int_{(R^{(k-1)})^c} \frac{\ell(R)^\alpha}{|z - x|^{m+\alpha}} |b_{(R^{(i)})^a}(z)| d\mu(z)
\lesssim 2^{-\alpha i/2},
\end{align*}
where
$$ \xi(z) = \Big[\frac{t}{(t + |z|)}\Big]^{m + \alpha}, \ \
\eta(z) = \frac{t^\alpha}{(t + |z-x|)^{m + \alpha}}\mathbf{1}_{E_2}(z)|b_{(R^{(i)})^a}(z)|.$$
\end{proof}

Next, we turn our attention to control the non-stopping term. Note that $R^{(i-1)}$ does not satisfy the stopping time condition $(1)$ and $(R^{(i-1)})^a =
(R^{(i)})^a$. Thus, we have
$$ |B_{R^{(i-1)}}| \lesssim \frac{1}{\mu(R^{(i-1)})} \Big|\int_{R^{(i-1)}} B_{R^{(i-1)}} b_{(R^{(i)})^a} d\mu \Big|
= \big| \langle \Delta_{R^{(i)}}f \rangle_{R^{(i-1)}} \big|.$$
According to Lemma $\ref{nested-2}$, we get
$$\bigg(\int_{\Rn} |\theta_t(\mathbf{1}_{(R^{(i-1)})^c} B_{R^{(i-1)}} b_{(R^{(i)})^a})(y)|^2 \Big(\frac{t}{t+|x-y|}\Big)^{m\lambda}
\frac{d\mu(y)}{t^m} \bigg)^{1/2}
\lesssim 2^{- \alpha i/2} \big| \langle \Delta_{R^{(i)}}f \rangle_{R^{(i-1)}} \big|.$$
Using the similar method that how we deduce the first inequality in $(\ref{G-nes'})$, we get
\begin{align*}
\mathcal{G}_{nes,out}^{''}
&\lesssim \sum_{i=r+1}^{s-\log_2 \ell(R)}  2^{-\frac{\alpha}{2}i} \bigg\| \mathbf{1}_{Q_0} \bigg(\sum_{k \leq s-i}\Big( \sum_{\substack{R \in
\mathcal{D}_{k,good} \\ R \subset Q^*}} \mathbf{1}_{R} \big| \langle \Delta_{R^{(i)}}f \rangle_{R^{(i-1)}} \big| \Big)^2 \bigg)^{1/2}\bigg\|_{L^p(\mu)} \\
&\leq \sum_{i=r+1}^{s-\log_2 \ell(R)}  2^{-\frac{\alpha}{2}i} \bigg\| \mathbf{1}_{Q_0} \bigg(\sum_{k \leq s-i}\Big( \sum_{S \in \mathcal{D}_{k+i-1}}
\frac{\mathbf{1}_{S}}{\mu(S)} \int_{S} |\Delta_{k+i}f(z) | d\mu(z) \Big)^2 \bigg)^{1/2}\bigg\|_{L^p(\mu)}.
\end{align*}
Because the rest of steps are still the same, we get $\mathcal{G}_{nes,out}^{''} \lesssim \mu(Q^*)^{1/p}$.

\vspace{0.3cm}
\noindent\textbf{$\bullet$ The stopping bound.}
Recall the construction of stopping cubes. From Lemma $\ref{nested-2}$, it follows that
$$\bigg(\int_{\Rn} \Big|\theta_t \Big(\mathbf{1}_{(R^{(i-1)})^c} \frac{\langle f \rangle_{R^{(i)}}}{\langle b_{(R^{(i)})^a} \rangle_{R^{(i)}}}
b_{(R^{(i)})^a}\Big)(y)\Big|^2 \Big(\frac{t}{t+|x-y|}\Big)^{m\lambda} \frac{d\mu(y)}{t^m} \bigg)^{1/2}
\lesssim 2^{- \alpha i/2} |\langle f \rangle_{R^{(i)}}|.$$
This indicates that
\begin{align*}
\mathcal{G}_{nes,in}^{''}
&\lesssim \bigg\| \mathbf{1}_{Q_0}\bigg(\sum_{\substack{R \in \mathcal{D}_{good}:R \subset Q^* \\ 2^{\zeta} < \ell(R) \leq 2^{s-r-1}}}
\mathbf{1}_{R}\Big(\sum_{i=r+1}^{s-\log_2 \ell(R)}\sum_{R:R^{(i-1)} \in \mathcal{F}_{Q^*}} 2^{- \alpha i/2} |\langle f \rangle_{R^{(i)}}| \Big)^2
\bigg)^{1/2}\bigg\|_{L^p(\mu)} \\
&\lesssim \bigg\| \mathbf{1}_{Q_0}\bigg(\sum_{i=r+1}^{s-\log_2 \ell(R)} 2^{- \alpha i/2} \sum_{S \in \mathcal{D}: S \subset Q^*}
|\langle f \rangle_{S}|^2 A_S \bigg)^{1/2}\bigg\|_{L^p(\mu)} \\
&\lesssim \Big\| \mathbf{1}_{Q_0}\Big( \sum_{S \in \mathcal{D}: S \subset Q^*} A_S \Big)^{1/2}\Big\|_{L^p(\mu)}
\end{align*}
where $$ A_S(x) = \sum_{\substack{S' \in ch(S) \\ (S')^a=S'}} \mathbf{1}_{S'}(x).$$
Furthermore, applying H\"{o}lder's inequality and Lemma $\ref{Carleson-sequence}$, we obtain
\begin{align*}
\mathcal{G}_{nes,in}^{''}
&\lesssim \Big\| \mathbf{1}_{Q^*} \Big( \sum_{S \in \mathcal{D}: S \subset Q^*} A_S \Big)^{p/2}\Big\|_{L^1(\mu)}^{1/p}
\lesssim \Big(\mu(Q^*)^{1/(2/p)'} \Big\| \sum_{S \in \mathcal{D}: S \subset Q^*} A_S \Big\|_{L^1(\mu)}^{p/2}\Big)^{1/p} \\
&\leq \bigg( \mu(Q^*)^{1-p/2}\Big( \sum_{S \in \mathcal{D}: S \subset Q^*} \sum_{\substack{S' \in ch(S) \\ (S')^a=S'}}\mu(S') \Big)^{p/2} \bigg)^{1/p} \\
&\lesssim \big( \mu(Q^*)^{1-p/2} \mu(Q^*)^{p/2} \big)^{1/p}
= \mu(Q^*)^{1/p}.
\end{align*}

So far, we have proved the inequality $(\ref{key-2})$ completely.
\qed

\section{Proof of $T1$ theorem $(\ref{key-1})$}\label{sec-T1}
In this section, our goal is to prove the inequality $(\ref{key-1})$. That is,
\begin{equation}\label{T1 theorem}
||g_{\lambda}^*(f)||_{L^p(\mu)} \lesssim (1+\mathcal{G}_{\text{glo}}(9))||f||_{L^p(\mu)}.
\end{equation}
It is worth mentioning that the proof of $T1$ theorem is partly similar to that of $(\ref{key-2})$. Moreover, martingale difference operators are much simpler
than before. So we only give some points different and key steps.
\subsection{Reduction to good cubes}
Without loss of generality, we may assume that $||g_{\lambda}^*||_{L^p(\mu) \rightarrow L^p(\mu)} < \infty$. Since
\begin{align*}
&\iint_{\R^{n+1}_+} \Big(\frac{t}{t+|x-y|}\Big)^{m\lambda} |\theta_t f(y)|^2 \frac{d\mu(y) dt}{t^{m+1}} \\
&=\lim_{s \rightarrow \infty}\sum_{\substack{R \in \mathcal{D}^\beta \\ \ell(R) \leq 2^s}} \mathbf{1}_{R}(x)\int_{\ell(R)/2}^{\ell(R)} \int_\Rn
\Big(\frac{t}{t+|x-y|}\Big)^{m\lambda} |\theta_t f(y)|^2 \frac{d\mu(y) dt}{t^{m+1}},
\end{align*}
the monotone convergence theorem implies that
\begin{align*}
&||g_{\lambda}^*(f)||_{L^p(\mu)} \\
&= \lim_{s \rightarrow \infty} \bigg[\int_{\Rn}\bigg(\sum_{\substack{R \in \mathcal{D}^\beta \\ \ell(R)\leq 2^s}} \mathbf{1}_{R}(x)\int_{\ell(R)/2}^{\ell(R)}
\int_\Rn \Big(\frac{t}{t+|x-y|}\Big)^{m\lambda} |\theta_t f(y)|^2 \frac{d\mu(y) dt}{t^{m+1}}\bigg)^{p/2} d\mu(x)\bigg]^{1/p}\\
&=\lim_{s \rightarrow \infty} \mathbb{E}_{\beta}\bigg[\int_{\Rn}\bigg(\sum_{\substack{R \in \mathcal{D}^\beta \\ \ell(R)\leq 2^s}}
\mathbf{1}_{R}(x)\int_{\ell(R)/2}^{\ell(R)} \int_\Rn \Big(\frac{t}{t+|x-y|}\Big)^{m\lambda} |\theta_t f(y)|^2 \frac{d\mu(y) dt}{t^{m+1}}\bigg)^{p/2}
d\mu(x)\bigg]^{1/p}.
\end{align*}
Assume that $f$ has compact support, $\supp f \subset [-2^N,2^N]^n$ and $||f||_{L^p(\mu)}=1$. It is sufficient to show that for any $s \geq N$ there holds that
\begin{align*}
&\mathbb{E}_{\beta}\bigg[\int_{\Rn}\bigg(\sum_{\substack{R \in \mathcal{D}^\beta \\ \ell(R)\leq 2^s}} \mathbf{1}_{R}(x)\int_{\ell(R)/2}^{\ell(R)} \int_\Rn
\Big(\frac{t}{t+|x-y|}\Big)^{m\lambda} |\theta_t f(y)|^2 \frac{d\mu(y) dt}{t^{m+1}}\bigg)^{p/2} d\mu(x)\bigg]^{1/p} \\
&\leq C (1+\mathcal{G}_{\text{glo}}(9)) + ||g_{\lambda}^*||_{L^p(\mu) \rightarrow L^p(\mu)}/2.
\end{align*}
Using the similar argument as in subsection $\ref{reduction-good}$, it is enough to prove
\begin{align*}
\bigg[\int_{\Rn}\bigg(\sum_{\substack{R \in \mathcal{D}_{good}^\beta \\ \ell(R)\leq 2^s}} \mathbf{1}_{R}(x)\int_{\ell(R)/2}^{\ell(R)} \int_\Rn
\Big(\frac{t}{t+|x-y|}\Big)^{m\lambda} |\theta_t f(y)|^2 \frac{d\mu(y) dt}{t^{m+1}}\bigg)^{\frac p2} d\mu(x)\bigg]^{\frac1p}
\lesssim (1+\mathcal{G}_{\text{glo}}(9)).
\end{align*}
\subsection{Reduction to paraproduct eatimate.}
We expend the fixed $f \in L^p(\mu)$ as follows
$$f=\lim_{\zeta \rightarrow \infty}\sum_{\substack{Q^* \in \mathcal{D}\\ \ell(Q^*)=2^s \\ Q^* \cap [-2^N,2^N]^n \neq \emptyset}} \sum_{\substack{Q \in
\mathcal{D} \\ Q \subset Q^* \\ \ell(Q)>2^{-\zeta}}}\Delta_Q f.$$
The martingale difference operators $\Delta_Q f = \sum_{Q' \in ch(Q)}\big(\langle f \rangle_{Q'} - \langle f \rangle_{Q}\big)\mathbf{1}_{Q'}$. If
$\ell(Q^*)=2^s$, $\Delta_{Q^*} f = \sum_{Q' \in ch(Q^*)}\langle f \rangle_{Q'}$.
Proceeding as we did in the subsection $\ref{reduction-martingale}$, we are reduced to proving that the quantity
\begin{align*}
\int_{\Rn}\bigg(\sum_{\substack{R \in \mathcal{D}_{good}\\ 2^{-\zeta} < \ell(R) \leq 2^s}} \mathbf{1}_{R}(x)\int_{\ell(R)/2}^{\ell(R)} \int_\Rn
\Big(\frac{t}{t+|x-y|}\Big)^{m\lambda}
\Big| \sum_{\substack{Q \in \mathcal{D}: Q \subset Q^* \\ \ell(Q)>2^{-\zeta}}}\theta_t(\Delta_Q f)(y) \Big|^2 \frac{d\mu(y) dt}{t^{m+1}}\bigg)^{\frac p2} d\mu(x)
\end{align*}
is dominated by $\big(1+\mathcal{G}_{\text{glo}}(9)\big)^p$.

As before, the splitting of the summation leads us to only bound $\mathcal{G}_{less}$, $\mathcal{G}_{sep}$, $\mathcal{G}_{adj}$ and $\mathcal{G}_{nes}$. It is
worth pointing out that the corresponding martingale estimate becomes
$$\Big\|\Big(\sum_{Q \subset Q^*}|\Delta_Q f|^2\Big)^{1/2}\Big\|_{L^p(\mu)} \lesssim ||f||_{L^p(\mu)} =1.$$
Applying the similar technique as in subsections $\ref{sec-less}$, $\ref{sec-sep}$ and $\ref{sec-adj}$, one can get $$\mathcal{G}_{less} + \mathcal{G}_{sep} +
\mathcal{G}_{adj} \lesssim 1.$$
Thus, it only remains to dominate $\mathcal{G}_{nes}$. However, as for the nested term, the only difference lies in the paraproduct estimate. In this setting,
paralleling to previous paraproduct estimate, it holds that
\begin{align*}
\mathcal{G}_{nes,par}^{''}
&= \bigg\|\bigg(\sum_{\substack{S \in \mathcal{D} \\ S \subset Q^*}}\sum_{\substack{R \in \mathcal{D}_{good} \\ R^{(r)}=S}}
|\langle f \rangle_S|^2 \mathbf{1}_{R} \int_{\ell(R)/2}^{\ell(R)} \int_\Rn \Big(\frac{t}{t+|x-y|}\Big)^{m\lambda}
|\theta_t \mathbf{1}(y)|^2 \frac{d\mu(y) dt}{t^{m+1}} \bigg)^{1/2}\bigg\|_{L^p(\mu)} \\
\end{align*}
In order to get $\mathcal{G}_{nes,par}^{''} \lesssim \mathcal{G}_{\text{glo}}(9)$, we need the following Carleson embedding theorem on $L^p(\mu)$. The general
version Carleson embedding theorem was shown in \cite{MM}.
\begin{proposition}
Let $\mathcal{D}$ be a dyadic grid in $\Rn$ and $1 < p \leq 2$ be a fixed number. Suppose that for every $S \in \mathcal{D}$ we have a function $A_S$
satisfying that $\supp A_S \subset S$ and
$$ \mathcal{C}ar := \bigg(\sup_{Q \in \mathcal{D}}\frac{1}{\mu(Q)}\int_{Q}\Big(\sum_{\substack{S \in \mathcal{D} \\ S \subset
Q}}|A_S(x)|^2\Big)^{p/2}d\mu(x)\bigg)^{1/p} < \infty.$$
Then we have that
$$ \Big\|\Big(\sum_{S \in \mathcal{D}}|\langle f \rangle_S|^2 |A_S|^2\Big)^{1/2}\Big\|_{L^p(\mu)}
\lesssim \mathcal{C}ar ||f||_{L^p(\mu)}.$$
\end{proposition}

Consequently, it suffices to prove
\begin{equation}\label{Car-Glo}
\mathcal{C}ar \lesssim \mathcal{G}_{\text{glo}}(9),
\end{equation}
where $A_S$ is defined by
$$ A_S(x)^2 = \sum_{\substack{R \in \mathcal{D}_{good} \\ R^{(r)}=S}} \mathbf{1}_{R}(x)\int_{\ell(R)/2}^{\ell(R)} \int_\Rn
\Big(\frac{t}{t+|x-y|}\Big)^{m\lambda} | \theta_t \mathbf{1}(y) |^2 \frac{d\mu(y) dt}{t^{m+1}}. $$
Notice that $R$ is a good cube and $\ell(Q) > 2^r \ell(R)$. Thus we have $R \subset Q$ and
$$ \dist(R,\partial Q) > \ell(R)^\gamma \ell(Q)^{1-\gamma}.$$
We set $\mathcal{W}_Q$ to be the maximal dyadic cubes $R \subset Q$ such that $2^r \ell(R) \leq \ell(Q)$ and
$\dist(R,\partial Q) \geq \ell(R)^\gamma \ell(Q)^{1-\gamma}$.
Hence, we gain the following inequality chain

\begin{align*}
\mathcal{H}
&:=\int_{Q}\Big(\sum_{\substack{S \in \mathcal{D} \\ S \subset Q}}|A_S(x)|^2\Big)^{p/2}d\mu(x) \\
&\leq \int_{Q}\bigg(\sum_{\substack{R \in \mathcal{D}_{good}:R \subset Q \\ 2^r \ell(R) \leq \ell(Q) \\
\dist(R,\partial Q) \geq \ell(R)^\gamma \ell(Q)^{1-\gamma}}}
 \mathbf{1}_{R}(x) \int_{\ell(R)/2}^{\ell(R)} \int_\Rn \Big(\frac{t}{t+|x-y|}\Big)^{m\lambda}
|\theta_t \mathbf{1}(y)|^2 \frac{d\mu(y) dt}{t^{m+1}} \bigg)^{\frac p2} d\mu(x) \\
&\leq \int_{Q}\bigg(\sum_{R \in \mathcal{W}_Q} \sum_{\substack{P \in \mathcal{D} \\ P \subset R}}
 \mathbf{1}_{P}(x) \int_{\ell(P)/2}^{\ell(P)} \int_\Rn \Big(\frac{t}{t+|x-y|}\Big)^{m\lambda}
|\theta_t \mathbf{1}(y)|^2 \frac{d\mu(y) dt}{t^{m+1}} \bigg)^{p/2} d\mu(x) \\
&\leq \int_{Q}\bigg(\sum_{R \in \mathcal{W}_Q}\mathbf{1}_{R}(x) \int_{0}^{\ell(R)} \int_\Rn \Big(\frac{t}{t+|x-y|}\Big)^{m\lambda}
|\theta_t \mathbf{1}(y)|^2 \frac{d\mu(y) dt}{t^{m+1}} \bigg)^{p/2} d\mu(x).
\end{align*}
Finally, it is important to note that $\{ 9R;R \in \mathcal{W}_Q\}$ has bounded overlaps (see \cite{CLX} Proposition 2.3 or \cite{LL} Proposition 3.4) and
$\{R; R \in \mathcal{W}_Q\}$ is a family of disjoint cubes.
\begin{align*}
\mathcal{H}
&\leq \sum_{R \in \mathcal{W}_Q}\int_{R}\bigg( \int_{0}^{\ell(R)} \int_\Rn \Big(\frac{t}{t+|x-y|}\Big)^{m\lambda}
|\theta_t \mathbf{1}(y)|^2 \frac{d\mu(y) dt}{t^{m+1}} \bigg)^{p/2} d\mu(x)\\
&\leq \mathcal{G}_{\text{glo}}(9)^p \sum_{R \in \mathcal{W}_Q} \mu(9R)
\lesssim \mathcal{G}_{\text{glo}}(9)^p \mu(Q).
\end{align*}
This shows $(\ref{Car-Glo})$. Thereby, we complete the proof of $(\ref{key-1})$.


\begin{thebibliography}{00}

\bibitem{AHMTT}P. Auscher, S. Hofmann, C. Muscalu, T. Tao, and C. Thiele, \emph{Carleson measures, trees, extrapolation, and $T(b)$ theorems}, Publ. Mat. 46 (2002), no. 2, 257-325.

\bibitem{CLX}M. Cao, K. Li, Q. Xue, \emph{A characterization of two weight norm inequality for Littlewood-Paley $g_{\lambda}^{*}$-function}, available at http://arxiv.org/abs/1504.07850.

\bibitem{C}M. Christ, \emph{A $T(b)$ theorem with remarks on analytic capacity and the Cauchy integral}, Colloq. Math. 60/61 (1990), 601-628.

\bibitem{Fe}C. Fefferman, {\emph Inequalities for strongly singular convolution operators}, Acta Math., 124 (1970), 9-36.

\bibitem{H2012}T. Hyt\"{o}nen, \emph{The sharp weighted bound for general Calder¡äon-Zygmund operators}, Ann. Math., 175 (3) (2012), 1473-1506.

\bibitem{H2014}T. Hyt\"{o}nen, \emph{The vector-valued nonhomogeneous $Tb$ theorem}, Int. Math. Res. Not. IMRN 2014(2) (2014) 451-511.

\bibitem{HM}T. Hyt\"{o}nen and H. Martikainen, \emph{On general local $Tb$ theorems}, Trans. Amer. Math. Soc. 364 (2012), no. 9, 4819-4846.

\bibitem{HM-1}T. Hyt\"{o}nen, H. Martikainen, \emph{Non-homogeneous Tb theorem and random dyadic cubes on metric measure space}, J. Geom. Anal. 22(4)(2012)1071-1107.

\bibitem{Hofmann}S. Hofmann, \emph{A proof of the local $Tb$ Theorem for standard Calder\'{o}n-Zygmund operators}, available at http://arxiv.org/abs/0705.0840.

\bibitem{LL}M. T. Lacey, K. Li, \emph{Two weight norm inequalities for $g$ function}, Math. Res. Lett., 21 (2014), no. 03, 521-536.

\bibitem{LM-CZ}M. Lacey and H. Martikainen, \emph{Local Tb theorem with $L^2$ testing conditions and general measures: Calder\'{o}n-Zygmund operators}, available at http://arxiv.org/abs/1310.8531.

\bibitem{LM-S}M. T. Lacey and H. Martikainen, \emph{Local Tb theorem with $L^2$ testing conditions and general measures: square functions}, available at http://www.arxiv.org/abs/1308.4571.

\bibitem{MM}H. Martikainen and M. Mourgoglou, \emph{Boundedness of non-homogeneous square functions and $L^q$ type testing conditions with $q \in (1,2)$}, http://arxiv.org/abs/1401.5457.

\bibitem{MMO}H. Martikainen, M. Mourgoglou, T. Orponen, \emph{Square functions with general measures II}, Indiana Univ. Math. J., to appear (2013), available at http://arxiv.org/abs/1305.6865.

\bibitem{MR} B. Muckenhoupt, R. L. Wheeden, \emph{Norm inequalities for the Littlewood-Paley function $g_\lambda^*$},  Trans. Amer. Math. Soc., 191 (1974), 95-111.

\bibitem{NTV2002}F. Nazarov, S. Treil, A. Volberg, \emph{Accretive system $Tb$-theorems on nonhomogeneous spaces}, Duke Math. J. 113 (2002), no. 2, 259-312.

\bibitem{NTV2003}F. Nazarov, S. Treil and A. Volberg, \emph{The Tb-theorem on non-homogeneous spaces}, Acta Math. 190 (2) (2003), 151-239.

\bibitem{Stein1961}E. M. Stein, \emph{On some function of Littlewood-Paley and Zygmund}, Bull. Amer. Math. Soc. 67(1961)99-101.

\bibitem{Stein1970}E. M. Stein, \emph{Topics in harmonic analysis related to the Littlewood-Paley theory}, Anna. Math. Stud., Princeton University Press, Princeton, N.J., 1970.


\end{thebibliography}
\end{document}